\documentclass[11pt]{article}
\usepackage{amsmath}
\usepackage{float}
\usepackage{amssymb}
\usepackage{amsthm}
\usepackage{graphicx}
\usepackage{color}
\usepackage{graphics}
\usepackage{multirow}
\usepackage[english]{babel}
\newtheorem{thm}{Theorem}[section]
\newtheorem{lem}[thm]{Lemma}
\newtheorem{dfn}[thm]{Definition}

\everymath{\displaystyle}
\numberwithin{equation}{section}

\title{ \bf Finite difference approximations for a size-structured population model with distributed states in the
recruitment}
\author{Azmy S. Ackleh$^1$\footnote{Corresponding author email: ackleh@louisiana.edu}\; J\'{o}zsef Z.
Farkas$^2$\footnote{Email address: jozsef.farkas@stir.ac.uk} \; Xinyu Li$^1$\footnote{Email address:
xxl0154@louisiana.edu}
\; and Baoling Ma$^1$\footnote{Email address: bxm4254@louisiana.edu}}
\date{\small
 {$^1$Department of Mathematics,\\
  University of Louisiana at Lafayette,\\
  Lafayette, LA 70504, USA.\\
  \vspace{0.2 in}
  $^2$Division of Computing Science and Mathematics,\\
  University of Stirling,\\
  Srirling, FK9 4LA, U.K.}
}

\begin{document}
\maketitle

\begin{quote}
{\bf Abstract:} In this paper we consider a size-structured population model where individuals may be recruited into
the population at different sizes. First and second order finite difference schemes are developed to approximate the
solution of the mathematical model. The convergence of the approximations to a unique weak solution with bounded
total variation is proved. We then show that as the distribution of the new recruits become concentrated at the
smallest size, the weak solution of the distributed states-at-birth model converges to the weak solution of the
classical Gurtin-McCamy-type size-structured model in the weak$^*$ topology. Numerical simulations are provided to
demonstrate the achievement of the desired accuracy of the two methods for smooth solutions as well as the superior
performance of the second-order method in resolving solution-discontinuities. Finally we provide an example where
supercritical Hopf-bifurcation occurs in the limiting single state-at-birth model and we apply the second-order
numerical scheme to show that such bifurcation occurs in the distributed model as well.  
\end{quote}

\baselineskip=18 pt

\section{\large Introduction}\label{}\medskip

Continuous structured population models are frequently used to study fundamental questions of population dynamics, see
e.g.~\cite{ADT, AF, AB, AMT, CS, JC1, JC2, Model, FH, MD, ST}. These models assume that individuals are distinguished
from one another by
characteristics such as body length, height, weight, maturity level, or age etc. These characteristics are often
referred to as size in general. In the classical one-dimensional case, size-structured models are formulated in terms
of a
nonlocal hyperbolic partial differential equation (PDE) describing the dynamics of the density $u(x,t)$ together with
an initial value $u_{0}(x)$ and a boundary condition at
$x=x_{0}$. Here $x$ is the structuring variable size. The boundary condition describes the inflow of newborns in the
population.
In most of these models, it is assumed that all the newborns have the same size $x_{0}$. It is clear in the case when
$x$ represents age and $x_{0}=0$. However, this assumption is not appropriate for many phenomena.
For example, newborns of human beings and other mammals can have different body lengths and weights at birth. In cell
populations, where large enough cells with different sizes will divide into two new daughter cells through mitosis and
cytokinesis, there is no fixed size for the newly-divided daughter cell when joining the population. Another example
comes from modeling fragmentation and coagulation in systems of reacting polymers where aggregates of different sizes
coalesce to form larger clusters or break apart into smaller ones~\cite{AF2,LM, Wrzosek}. In all of these examples, the
recruitment cannot be accurately modeled by simply imposing one boundary condition at the $x_0$.

Population models with distributed states-at-birth thus were introduced and studied for example in~\cite{AF,CS, Model,
ST}. In \cite{CS} the authors considered a very general size-structured model where individuals
may be recruited into the population at different sizes. The recruitment of new individuals is demonstrated in the
partial differential equation and modeled by a Lipschitz operator. They studied well-posedness of the model and
established global existence and uniqueness of solutions utilizing results from the theory of nonlinear evolution
equations. In \cite{ST} the authors studied an age-size-structured population model which assumes that size-at-birth is
distributed. The authors proved the existence of unique solutions to the model using a contraction mapping argument.
The local asymptotic stability of equilibria is also discussed using results from the theory of strongly continuous
semigroups of bounded linear operators. Distributed recruitment terms also appear in structured population models
dealing with cell division~\cite{MD} and in modeling reacting polymers by means of fragmentation models~\cite{MLM}.

In this paper we consider the following nonlinear Gurtin-MacCamy type model with a distributed recruitment
term (see, e.g.,~\cite{Model}). In what follows we will use the abbreviation DSSM when referring to the model below.
\begin{equation}\label{model}
\begin{array}{ll}
\frac{\partial}{\partial t}p(s,t)+\frac{\partial}{\partial s}\left(\gamma(s,Q(t))p(s,t)\right) & \\
\hspace{1 in} =-\mu(s,Q(t))p(s,t)
+\int_{0}^{1} \beta(s,y,Q(t))p(y,t)dy,
& s\in (0,1), t\in (0,T), \\
\gamma(0,Q(t))p(0,t)=0, & t\in[0,T],\\
p(s,0)=p^{0}(s), & s\in[0,1].\\
\end{array}
\end{equation}
Here, $p(s,t)$ denotes the density of individuals of size $s$ at time $t$. Therefore, $Q(t)=\int_{0}^{1}p(s,t)ds$
provides the total population at time $t$. The functions $\gamma$ and $\mu$ represent the individual growth and
mortality rate, respectively. It is assumed that individuals may be recruited into the population at different sizes
with $\beta(s,y,Q)$ being the rate at which one individual of size $y$ gives birth to an individual of size $s$ when
the total population is $Q$. There is no-inflow of individuals
through the boundary $s= 0$ since $p(0, t) = 0$ for all $t \geq 0$.

In \cite{Model} the authors analyzed the asymptotic behavior of solutions of model~(\ref{model}) using positive
perturbation arguments and results from the spectral theory of positive semigroups. In \cite{AF} the question of the
existence of non-trivial steady states is studied based on the reformulation of the problem~(\ref{model}) as an
integral equation. However, to our knowledge, numerical schemes for computing approximate solutions of the
distributed-rate model~(\ref{model}) have not been developed. Thus, in this paper we focus on the development of finite
difference schemes to approximate the solution of model~(\ref{model}). Efficient schemes are essential for solving
optimal control problems or parameter estimation problems governed by model~(\ref{model}) as such problems require
solving the model numerous times before an optimal solution is obtained.

Furthermore, we establish a connection between the model~(\ref{model}) and the following classical size-structured
model which will be referred to as CSSM for abbreviation.
\begin{equation}\label{model_McKendrick}
\begin{array}{ll}
\frac{\partial}{\partial t}p(s,t)+\frac{\partial}{\partial s}\left(\gamma(s,Q(t))p(s,t)\right)=-\mu(s,Q(t))p(s,t),
& s\in (0,1), \quad t\in (0,T), \\
\gamma(0,Q(t))p(0,t)=\int_{0}^{1} \tilde \beta(y,Q(t))p(y,t)dy, & t\in[0,T],\\
p(s,0)=p^{0}(s), & s\in[0,1].\\
\end{array}
\end{equation}\\
Here $\tilde \beta$ is the fertility rate of individuals of size $y$ at population level $Q$ and the rest of the
functions and parameters have similar interpretations as in model~(\ref{model}). We show that as the distribution of
the new recruits become concentrated at the smallest size, the weak solution of \eqref{model} converge in the weak*
topology to the weak solution of \eqref{model_McKendrick}. To our knowledge, this is the first theoretical result that
connects the two models.

This paper is organized as follows. Assumptions and notation are introduced in Section 2. In Section 3, we present a
first order explicit upwind scheme for solving the DSSM and prove its convergence to a unique weak solution with
bounded total variation. In Section 4 we present a second order explicit finite difference scheme and prove its
convergence. In Section 5 we establish the connection between DSSM and CSSM. Section 6 is devoted to numerical
simulations and to the construction of a simple example in which supercritical Hopf-bifurcation occurs. We give
concluding remarks in Section 7.

\section{\large Assumptions and notation}\label{sec_assumption}

\noindent
Let $\mathbb{D}_1=[0,1]\times[0,\infty)$ and ${\mathbb{D}_2}=[0,1]\times[0,1]\times[0,\infty)$.
Let $c$ be a sufficiently large positive constant. Throughout the paper we impose the following regularity conditions
on the functions involved in the DSSM.
\begin{itemize}
\item[(H1)]  $ \gamma(s,Q)$ is continuously differentiable with respect to $s$ and $Q$, $\gamma_{s}(s,Q)$ and
    $\gamma_{Q}(s,Q)$ are Lipschitz continuous in $s$ with Lipschitz constant $c$, uniformly in $Q$. Moreover,
    $0<\gamma(s,Q)\leq c$ for $s\in[0,1)$ and $\gamma(1,Q)=0$.
\item[(H2)] $ 0\leq \mu(s,Q)\leq c$, $\mu$ is Lipschitz continuous in $s$ and $Q$ with Lipschitz constant $c$.
\item[(H3)]$ 0\leq \beta(s,y,Q)\leq c $, $\beta(s,y,Q)$ is Lipschitz continuous in $Q$ with Lipschitz constant $c$,
    uniformly in $s$ and $y$. Moreover, for every partition $\{{s_i}\}_{i=1}^N$ of $[0,1]$, we have
\begin{equation*}
\sup_{(y,Q) \in [0,1]\times[0,\infty)} \sum_{i=1}^{N}|\beta(s_{i}, y, Q)-\beta(s_{i-1}, y, Q)|\leq c.
\end{equation*}
\item[(H4)]  $p^0 \in BV([0,1])$, where $BV$ stands for the space of functions with bounded total variation, and $
    p^0(s)\geq0$.
\end{itemize}
Now we give the definition of a weak solution to the DSSM as follows.
\begin{dfn}
A function $p\in BV ([0,1]\times[0,T] )$ is called a weak solution of the DSSM model \eqref{model} if it satisfies:
\begin{equation}\label{weak solution of DSSM}
\begin{array}{ll}
\int _0^1p(s,t)\phi(s,t)ds-\int _0^1p^0(s)\phi(s,0)ds\\
=\int _0^t\int _0^1  p(s,\tau)[\phi _\tau(s,\tau) +\gamma (s,Q(\tau))\phi _s(s,\tau)-\mu(s,Q(\tau))\phi(s,\tau)]dsd\tau
\\
\quad + \int _0^t\int _0^1\int _0^1 \beta(s,y,Q(\tau))p(y,\tau)\phi(s,\tau)dydsd\tau
\end{array}
\end{equation}
for every test function $\phi \in C^1\left( [0,1] \times [0,T]\right)$ and $t \in [0,T]$.\\
\end{dfn}
\noindent
Suppose that the intervals $[0,1]$ and $[0,T]$ are divided into $N$ and $L$ subintervals,
respectively. The following notations will be used throughout the
paper: $\Delta s=1/N$ and $\Delta t=T/L$. The discrete mesh points are given by $s_{i}=
i\Delta s$, $t_{k}= k\Delta t$ for $i=0,1,\cdots, N$, $k=0,1, \cdots, L$. For ease of
notation, we take a uniform mesh with constant sizes $\Delta s$ and $\Delta t$. More general nonuniform meshes can be
similarly
considered. We shall denote by $p_{i}^k$ and $Q^k$ the finite difference approximation of $p(s_{i},t_{k})$ and
$Q(t_{k})$, respectively.
We also let
\begin{equation*}
\begin{array}{ll}
\gamma_{i}^k=\gamma(s_{i}, Q^k),\hspace{0.14in}
\mu_{i}^k=\mu(s_{i},Q^k),\hspace{0.14in}
\beta_{i,j}^k=\beta(s_{i}, y_j, Q^{k}).\\
\end{array}
\end{equation*}
We define the $\ell^{1}$, $\ell^{\infty}$ norms and the TV (total variation) seminorm of the grid functions $p^k$ by
$$
\|p^k\|_{1}=\sum_{i=1}^{N}|p_{i}^k|\Delta s,\hspace{0.4in}
\|p^k\|_{\infty}=\max_{0\leq i\leq N}|p_{i}^k|,\hspace{0.4in}
TV(p^k)=\sum_{i=0}^{N-1}{|p_{i+1}^k-p_i^k|},
$$
and the finite difference operators by
\begin{eqnarray*}
\Delta_{+} p_{i}^{k}=p_{i+1}^{k}-p_{i}^{k}, \hspace{5 pt} 0\leqslant i \leqslant N-1,\quad\quad
\Delta_{-}p_{i}^{k}=p_{i}^{k}-p_{i-1}^{k},\hspace{5 pt} 1\leqslant i\leqslant N.
\end{eqnarray*}
Throughout the discussion, we impose the following CFL condition concerning $\Delta s$ and $\Delta t$:\\
$\begin{array}{ll}
(H5)\quad c \frac{3\Delta t}{2\Delta s}+ c\Delta t\leqslant 1.
\end{array}$

\section{\large A first order upwind scheme}\label{sec_upwind}\medskip
\noindent
We first discretize model~(\ref{model}) using the following first order explicit upwind scheme:
\begin{equation}\label{first order 1}
\begin{array}{ll}
\begin{array}{ll}
\frac{p_i^{k+1}-p_i^k}{\Delta t}+\frac{\gamma _i^k p_i^k -\gamma_{i-1}^{k}p_{i-1}^{k}}{\Delta s} = -\mu
_i^kp_i^k+\sum_{j=1}^{N} \beta_{i,j}^{k}p_{j}^{k} \Delta s, & 1\leq i\leq N,\quad 0 \leq k \leq L-1,\\
\gamma_ 0^k p_0^k=0, & 0 \leq k \leq L,\\
p_i^0=p^0(s_i), & 0 \leq i \leq N,\\
\end{array}
\end{array}
\end{equation}\\
where the total population $Q^{k}$ is discretized by a right-hand sum $Q^{k}=\sum_{i=1}^{N} p_{i}^{k} \Delta s$.\\
\noindent
We can equivalently write the first equation in~(\ref{first order 1}) as follows:
\begin{equation}\label{first order 2}
p_i^{k+1}=\frac{\Delta t}{\Delta s} \gamma _{i-1}^{k} p_{i-1}^k +\left(1-\frac{\Delta t}{\Delta s} \gamma _i^k -\mu
_i^k \Delta t\right)p_i^k + \left(\sum_{j=1}^{N} \beta _{i,j}^{k} p_j^k \Delta s \right) \Delta t,\hspace{20pt} 1 \leq
i \leq N, \quad 0 \leq k \leq L-1.\\
\end{equation}
The boundary condition $ \gamma(0,Q(t))p(0,t)=0 $ and assumption (H1)
imply that $p_{0}^{k}=0$ for $k\geq 0$. One can easily see that under assumptions (H1)-(H5), $p_i^{k+1}\geq 0$, for
$i=1,2,\cdots, N$ and $k=0,1,\cdots, L-1$. Therefore, the scheme~(\ref{first order 1}) has a unique nonnegative
solution.\\
\subsection{Estimates for the first order finite difference scheme}
In this section we use techniques similar to \cite{AI_1997,J_Smoller}. We begin by establishing an $\ell_1$ bound on
the approximations.
\begin{lem}\label{lem_FO1} The following estimate holds:
\begin{equation*}
\begin{array}{ll}
\Vert p^k \Vert _1 \leq \left(1+ c\Delta t\right) ^k \Vert p^0 \Vert _1 \leq \left(1+ c\Delta t\right) ^L \Vert p^0
\Vert _1\leq \exp(cT)\Vert p^0 \Vert _1\equiv M_1, \quad k=0,1,\cdots, L. \\
\end{array}
\end{equation*}
\end{lem}
\begin{proof}
Multiplying~(\ref{first order 2}) by $\Delta s$ and summing over $i=1,2,\cdots ,N$, we have
\begin{equation*}
\sum_{i=1}^{N}p_i^{k+1} \Delta s =\sum_{i=1}^{N}p_i^k\Delta s-\Delta t\sum_{i=1}^{N} (\gamma _i^k p_i^k
-\gamma_{i-1}^{k}p_{i-1}^{k})-\sum_{i=1}^{N}p_i^k\mu_i^k \Delta s \Delta
t+\sum_{i=1}^{N}\left(\sum_{j=1}^{N}\beta_{i,j}^k p_j^k\Delta s \right) \Delta s \Delta t.
\end{equation*}
Therefore by assumptions (H1)-(H3) and the second equation in~(\ref{first order 1})
\begin{equation*}
\begin{array}{ll}
\Vert p^{k+1}\Vert_1 &\leq \Vert p^k \Vert _1 -\Delta t \left(\gamma_N^k p_N^k -\gamma_0^k p_0^k \right)+c \Vert p^k
\Vert _1 \Delta t\\
&=(1+c \Delta t) \Vert p^k \Vert _1,
\end{array}
\end{equation*}
which then implies the estimate.
\end{proof}
Note that $Q^k=\sum_{i=1}^{k}p_i^k\Delta s=\|p^k\|_1\leq M_1$. We now define $\mathbb{D}_3=[0,1]\times[0, M_1]$.
\begin{lem}\label{lem_FO2}
The following estimate holds:
\begin{equation*}
\begin{array}{ll}
\Vert p^k \Vert _\infty  \leq \left(1+ 2c\Delta t\right) ^k \Vert p^0 \Vert _\infty \leq \left(1+ 2c\Delta t\right) ^L
\Vert p^0 \Vert _\infty\leq \exp(2cT)\Vert p^0 \Vert _\infty, \quad k=0,1,\cdots, L.
\end{array}
\end{equation*}
\end{lem}
\begin{proof}
Since $p_0^k=0$ for $k\geq0$, $\Vert p^{k+1} \Vert _\infty$ is obtained at $p_i^{k+1}$ for some $1 \leq i \leq N $.\\
From~(\ref{first order 2}) and assumptions (H1), (H3) and (H5) we have
\begin{equation*}
\begin{array}{ll}
\Vert p^{k+1} \Vert _{\infty} &\leq \frac{\Delta t}{\Delta s} \gamma_{i-1}^k \Vert p^k \Vert _{\infty} +\left(
1-\frac{\Delta t}{\Delta s} \gamma_i^k -\mu_i^k \Delta t \right) \Vert p^k \Vert_{\infty}+c \Vert
p^k\Vert_{\infty}\Delta t \\
&\leq \Vert p^k\Vert _{\infty} +\sup_{\mathbb{D}_3}\vert \gamma_s \vert \Vert p^k \Vert _{\infty} \Delta t +c \Vert p^k
\Vert _{\infty}\Delta t\\
&\leq (1+2c\Delta t) \Vert p^k \Vert _{\infty}.
\end{array}
\end{equation*}
\end{proof}

\begin{lem} \label{lem_FO3}
There exists a positive constant $M_2$ such that
$TV (p^k)\leq M_2$, $k=0,1,\cdots,L$.
\end{lem}
\begin{proof}
From the first equation in~(\ref{first order 1}) we have
\begin{equation*}
\begin{array}{ll}
p_{i+1}^{k+1}-p_i^{k+1} &= \left(p_{i+1}^k-p_i^k \right)-\frac{\Delta t}{\Delta s}\left[\left(\gamma_{i+1}^k p_{i+1}^k
-\gamma_i^k p_i^k \right)-\left(\gamma_i^k p_i^k -\gamma_{i-1}^k p_{i-1}^k \right) \right]\\
&\quad -\Delta t\left(\mu_{i+1}^k p_{i+1}^k -\mu_i^k p_i^k \right) +\sum_{j=1}^{N}\left(\beta_{i+1,j}^k -\beta_{i,j}^k
\right)p_j^k \Delta s \Delta t.
\end{array}
\end{equation*}
Simple calculations yield
\begin{equation*}
\begin{array}{ll}
\left(\gamma_{i+1}^k p_{i+1}^k -\gamma_i^k p_i^k \right)-\left(\gamma_i^k p_i^k -\gamma_{i-1}^k p_{i-1}^k \right) \\
 =\gamma_{i+1}^k\left(p_{i+1}^k -p_i^k \right)+\left(\gamma_{i+1}^k -\gamma_i^k \right) p_i^k -\gamma_i^k \left( p_i^k
 - p_{i-1}^k \right)-\left(\gamma_i^k-\gamma_{i-1}^k \right)p_i^k \\
 =\gamma_{i+1}^k \left(p_{i+1}^k -p_i^k \right)-\gamma_i^k\left(p_i^k -p_{i-1}^k \right)+\left(\gamma_i^k
 -\gamma_{i-1}^k \right)\left(p_i^k - p_{i-1}^k \right)\\
 \quad + \left[\left(\gamma_{i+1}^k -\gamma_i^k \right)-\left(\gamma_i^k-\gamma_{i-1}^k \right) \right] p_i^k.
\end{array}
\end{equation*}
Therefore, for $1\leq i\leq N-1$
\begin{equation}\label{p_difference}
\begin{array}{ll}
p_{i+1}^{k+1}-p_i^{k+1} &=\left(1-\frac{\Delta t}{\Delta s}\gamma_{i+1}^k \right)\left(p_{i+1}^k -p_i^k
\right)+\frac{\Delta t}{\Delta s} \gamma_i^k\left(p_i^k-p_{i-1}^k \right)-\frac{\Delta t}{\Delta
s}\left(\gamma_i^k-\gamma_{i-1}^k \right)\left(p_i^k -p_{i-1}^k \right) \\
&\quad -\frac{\Delta t}{\Delta s}\left[ \left(\gamma_{i+1}^k -\gamma_i^k \right)-\left(\gamma_i^k-\gamma_{i-1}^k
\right) \right]p_i^k -\Delta t \left(\mu_{i+1}^k p_{i+1}^k -\mu_i^k p_i^k \right)\\
&\quad +\sum_{j=1}^N \left(\beta_{i+1,j}^k -\beta_{i,j}^k \right) p_j^k \Delta s \Delta t.
\end{array}
\end{equation}
Summing~(\ref{p_difference}) over $i=0,1,\cdots, N-1$ and applying assumptions (H1) and (H5) we arrive at
\begin{equation}
\begin{array}{ll}\label{TV_1}
TV(p^{k+1}) &= \vert p _{1}^{k+1}-p _0^{k+1} \vert + \sum _{i=1}^{N-1} \vert p_{i+1}^{k+1} - p_i^{k+1} \vert \\
&= p_1^{k+1} + \sum_{i=1}^{N-1} \vert p_{i+1}^k-p_i^k \vert -\frac{\Delta t}{\Delta s} \sum_{i=1}^{N-1}
\left(\gamma_{i+1}^k \vert p_{i+1}^k -p_i^k\vert -\gamma_i^k \vert p_i^k -p_{i-1}^k \vert \right)\\
&\quad+\frac{\Delta t}{\Delta s}\sum_{i=1}^{N-1}\vert \gamma_i^k -\gamma_{i-1}^k \vert \vert p_i^k -p_{i-1}^k
\vert+\frac{\Delta t}{\Delta s}\sum_{i=1}^{N-1} \left| \left( \gamma_{i+1}^k-\gamma_i^k \right) -\left(\gamma_i^k
-\gamma_{i-1}^k \right) \right| p_i^k\\
&\quad+ \Delta t \sum_{i=1}^{N-1} \left| \mu_{i+1}^kp_{i+1}^k-\mu_i^k p_i^k \right|+\sum_{i=1}^{N-1} \sum_{j=1}^N
\left|\beta_{i+1,j}^k-\beta_{i,j}^k  \right| p_j^k \Delta s \Delta t.
\end{array}
\end{equation}
By~(\ref{first order 2}) and assumptions (H1)-(H3),
\begin{equation}\label{FO3_1}
\begin{array}{ll}
p_1^{k+1}&=\frac{\Delta t}{\Delta s} \gamma_0^k p_0^k+\left(1-\frac{\Delta t}{\Delta s} \gamma_1^k-\mu_1^k \Delta t
\right) p_1^k+\left(\sum_{j=1}^{N}\beta_{1,j}^k p_j^k \Delta s \right) \Delta t \\
&\leq p_1^k -\frac{\Delta t}{\Delta s} \gamma_1^k p_1^k +c\| p^k\|_1 \Delta t.
\end{array}
\end{equation}
It can be seen from assumption (H1) that
\begin{equation}\label{FO3_2}
\begin{array}{ll}
\sum_{i=1}^{N-1}\left( \gamma_{i+1}^k | p_{i+1}^k -p_i^k|-\gamma_i^k| p_i^k-p_{i-1}^k|
\right)=\gamma_N^k|p_N^k-p_{N-1}^k |-\gamma_1^k |p_1^k-p_0^k|=-\gamma_1^kp_1^k,
\end{array}
\end{equation}
and
\begin{equation}\label{FO3_3}
\begin{array}{ll}
\frac{\Delta t}{\Delta s} \sum_{i=1}^{N-1} |(\gamma_{i+1}^k -\gamma_i^k )-(\gamma_i^k -\gamma_{i-1}^k)|p_i^k\\
=\frac{\Delta t}{\Delta s}\sum_{i=1}^{N-1} |\gamma_s(\hat{s}_{i+1},Q^k)-\gamma_s(\hat{s}_i,Q^k) | p_i^k \Delta s\\
\leq \Delta t \sum_{i=1}^{N-1} 2c p_i^k\Delta s \leq 2c \| p^k \|_1 \Delta t,
\end{array}
\end{equation}
where $\hat{s}_{i}\in [s_{i-1}, s_{i}]$ and $\hat{s}_{i+1}\in [s_i, s_{i+1}]$.\\
By assumption (H2),
\begin{equation}\label{FO3_4}
\begin{array}{ll}
\sum_{i=1}^{N-1}\left|\mu_{i+1}^k p_{i+1}^k -\mu_i^k p_i^k \right| \Delta t\\
\leq\Delta t \sum_{i=1}^{N-1} |\mu_{i+1}^k -\mu_i^k| p_{i+1}^k+\sum_{i=1}^{N-1} \sup_{\mathbb{D}_3} \mu |p_{i+1}^k
-p_i^k | \Delta t\\
\leq c \| p^k \|_1 \Delta t +c \sum_{i=1}^{N-1} |p_{i+1}^k -p_i^k | \Delta t.
\end{array}
\end{equation}
By assumption (H3),
\begin{equation}\label{FO3_5}
\begin{array}{ll}
\sum_{i=1}^{N-1}\sum_{j=1}^{N} \left| \beta_{i+1,j}^k -\beta_{i,j}^k \right| p_j^k \Delta s \Delta t
=\sum_{j=1}^{N}\left( \sum_{i=1}^{N-1} \left| \beta_{i+1,j}^k -\beta_{i,j}^k \right|\right) p_j^k \Delta s \Delta t
=c \| p^k \| _1 \Delta t.
\end{array}
\end{equation}
A combination of~(\ref{TV_1})-(\ref{FO3_5}) then yields
\begin{equation}
\begin{array}{ll}
TV(p^{k+1})&\leq p_1^k -\frac{\Delta t}{\Delta s}\gamma_1^k p_1^k+c\|p^k\|_1 \Delta t+ \sum_{i=1}^{N-1}|p_{i+1}^k
-p_i^k |+\frac{\Delta t}{\Delta s} \gamma_1^k p_1^k+\frac{\Delta t}{\Delta s}|\gamma_s| \Delta s\sum_{i=1}^{N-1} |p_i^k
-p_{i-1}^k|\\
& \quad +c\| p^k\|_1\Delta t+c\| p^k\|_1\Delta t+c\sum_{i=1}^{N-1}|p_{i+1}^k -p_i^k |\Delta t+2c\| p^k\|_1 \Delta t.
\end{array}
\end{equation}
Therefore, from assumption (H1), Lemmas~\ref{lem_FO1} and~\ref{lem_FO2} there exist positive constants $c_1$ and $c_2$
such that
\begin{equation*}
\begin{array}{ll}
TV(p^{k+1})\leq (1+c_1 \Delta t) TV(p^k) +c_2 \Delta t,
\end{array}
\end{equation*}
which leads to the desired result.
\end{proof}

\begin{lem} \label{lem_FO4}
There exists a positive constant $M_3$ such that for any $q_1 > q_2>0$ the following estimate holds:
$$\sum_{i=1}^{N}\left| \frac{p_i^{q_1}-p_i^{q_2}}{\Delta t}\right| \Delta s \leq M_3 (q_1-q_2).$$
\end{lem}
\begin{proof}
By~(\ref{first order 2}) and assumptions (H1)-(H3) we have
\begin{equation*}
\begin{array}{ll}
\sum_{i=1}^{N}\left| \frac{p_i^{k+1}-p_i^k}{\Delta t} \right| \Delta s
&=\sum_{i=1}^N\left|\gamma_{i-1}^kp_{i-1}^k -\gamma_i^k p_i^k -\mu_i^k p_i^k \Delta s + \sum_{j=1}^N\beta_{i,j}^k p_j^k
\Delta s \Delta s \right|\\
&  \leq \sum_{i=1}^N|\gamma_{i}^k-\gamma_{i-1}^k|p_{i-1}^k+\sum_{i=1}^N
\gamma_{i}^{k}|p_{i}^{k}-p_{i-1}^{k}|+\sum_{i=1}^N \mu_{i}^{k}p_i^k\Delta s\\
&\quad +\sum_{i=1}^N\sum_{i=j}^N \beta_{i,j}^{k}p_{j}^{k}\Delta s\Delta s\\
& \leq c TV(p^k)+3c\|p^k\|_1.
\end{array}
\end{equation*}
Thus, by Lemmas~\ref{lem_FO1} and~\ref{lem_FO3} there exists a positive constant $M_3$ such that
 $$\sum_{i=1}^{N}\left| \frac{p_i^{k+1}-p_i^{k}}{\Delta t}\right|\Delta s\leq M_3.$$
Therefore,
\begin{equation*}
\begin{array}{ll}
\sum_{i=1}^N \left| \frac{p_i^{q_1}-p_i^{q_2}}{\Delta t} \right| \Delta s
\leq \sum_{i=1}^N \sum_{k=q_2}^{q_1}\left| \frac{p_i^{k+1}-p_i^k}{\Delta t} \right| \Delta s
\leq  M_3 (q_1-q_2).
\end{array}
\end{equation*}
\end{proof}

\subsection{Convergence of the difference approximations to a unique weak solution}
Following similar notation as in \cite{J_Smoller} we define a set of functions $\lbrace P_{\Delta s,\Delta t}\rbrace$
by $ \{ P_{\Delta s,\Delta t}(s,t)\}=p_i^k$ for $s\in [s_{i-1},s_i), t\in [t_{k-1},t_k)$, $i=1,2,\cdots,N$, and
$k=1,2,\cdots,L$. Then by the Lemmas~\ref{lem_FO1} to~\ref{lem_FO4}, the set of functions $\lbrace P_{\Delta s,\Delta
t}\rbrace$ is compact in the topology of $\mathcal{L}^1((0,1)\times (0,T))$. Hence, following the proof of Lemma $16.7$
on page $276$ in~\cite{J_Smoller} we obtain the following result.
\begin{thm}\label{theorem U_convergence1}
There exists a subsequence of functions $\lbrace P_{\Delta s_r,\Delta t_r}\rbrace \subset \lbrace P_{\Delta s,\Delta
t}\rbrace$ which converges to a function $p \in BV\left( [0,1] \times [0,T] \right) $ in the sense that for all $t>0$,
$$\int _0^1 \vert P_{\Delta s_r,\Delta t_r}-p(s,t)\vert ds\longrightarrow 0,$$
$$\int _0^T \int_0^1 \vert P_{\Delta s_r,\Delta t_r}-p(s,t)\vert ds dt \longrightarrow 0$$
as $r \rightarrow \infty $ (i.e., $\Delta a_r, \Delta s _r, \Delta t_r \rightarrow 0$). Furthermore, there exists
a constant $M_4$ depending on $\Vert p^0 \Vert_{BV\left( [0,1] \times [0,T] \right)}$ such that the limit function
satisfies
$$ \Vert p \Vert_{BV\left( [0,1] \times [0,T] \right)}\leq M_4.$$
\end{thm}

We show in the next theorem that the limit function $p(s,t)$ constructed by the finite difference scheme is a weak
solution of the DSSM model \eqref{model}.
\begin{thm}\label{thm_uniqueness}
The limit function $p(s,t)$ defined in Theorem~\ref{theorem U_convergence1} is a weak solution of
problem~(\ref{model}). Moreover, it satisfies
$$\Vert p \Vert _{L^\infty\left( (0,1) \times (0,T) \right)}\leq \exp(2cT) \Vert p^{0} \Vert _\infty.$$
\end{thm}
\begin{proof}
The fact that $p(s,t)$ is a weak solution with bounded total variation follows from Lemma~\ref{lem_FO1}-\ref{lem_FO4}
and Lemma 16.9 on page 280 of~\cite{J_Smoller}. The bound on $\Vert p \Vert _{L^\infty\left( (0,1) \times (0,T)
\right)}$ is obtained by taking the limit in the bounds
of the difference approximation in Lemma~\ref{lem_FO2}.
\end{proof}

The following theorem guarantees the continuous dependence of the solution $p_i^k$ of~(\ref{first order 1}) with
respect to the initial condition $p_i^0$.
\begin{thm}\label{theorem_WS1}
Let $\left\lbrace p_i^k\right\rbrace$ and $\left\lbrace\hat{p}_i^k\right\rbrace$ be solutions of~(\ref{first order 1})
corresponding to the initial conditions $\left\lbrace p_i^0 \right\rbrace$ and $\left\lbrace\hat{p}_i^0\right\rbrace$,
respectively. Then there exists a positive constant $\delta$ such that
$$\Vert p^{k+1}-\hat{p}^{k+1} \Vert _1 \leq (1+\delta t) \Vert p^k-\hat{p}^k \Vert _1,\quad\quad \text{for all}\quad
k\geq 0.$$
\end{thm}
\begin{proof}
Let $u_i^k=p_i^k-\hat{p}_i^k$ for $i=0,1,\cdots,N$ and $k=0,1,\cdots,L$. Then by~(\ref{first order 2}) $u_i^k$
satisfies
\begin{equation}\label{eq_u}
\begin{array}{ll}
u_i^{k+1}
=\frac{\Delta t}{\Delta s}\left(\gamma _{i-1}^k p_{i-1}^k -\hat{\gamma}_{i-1}^{k}\hat{p}_{i-1}^{k}\right)+\left( p_i^k
-\hat{p}_i^k\right)-\frac{\Delta t}{\Delta s} \left( \gamma _i^k p_i^k -\hat{\gamma}_i^k \hat{p}_i^k \right) \\
\quad\quad\quad -\Delta t\left( \mu _i^k p_i^k -\hat{\mu}_i^k \hat{p}_i^k\right)+\sum_{j=1}^{N} \left(\beta_{i,j}^k
p_j^k-\hat{\beta}_{i,j}^{k} \hat{p}_j^k\right) \Delta s \Delta t, \quad \quad 1 \leq i \leq N, \quad
0\leq k\leq L-1,\\
u_0^{k+1} = p_0^{k+1}-\hat{p}_0^{k+1}=0,\quad \quad 0\leq k\leq L-1.
\end{array}
\end{equation}
Here $\hat{Q}^k=\sum_{i=1}^{N}\hat{p}_i^k$, $\hat{\gamma}_i^k = \gamma(s_i,t_k,\hat{Q}^k)$ and similar notations are
used for $\hat{\mu}_i^k$ and $\beta_{i,j}^k$. Using the first equation of~(\ref{eq_u}) and assumption $(H5)$ we obtain
\begin{equation*}
\begin{array}{ll}
\vert u_i^{k+1} \vert &\leq \left( 1-\frac{\Delta t}{\Delta s} \gamma _i^k -\Delta t \mu _i^k\right) \vert u_i^k \vert
+\frac{\Delta t}{\Delta s} \gamma _{i-1}^k \vert u_{i-1}^k \vert + \Delta t\vert \left( \gamma _{i-1}^k -\hat{\gamma}
_{i-1}^k \right)\hat{p}_{i-1}^k- \left( \gamma _i^k -\hat{\gamma} _i^k \right) \hat{p}_i^k \vert \\
&\quad + \Delta t \vert \mu _i^k -\hat{\mu} _i^k \vert \hat{p}_i^k +\sum_{j=1}^N \beta _{i,j}^k \vert \mu _i^k \vert
\Delta s \Delta t +\sum_{j=1}^N \vert \beta_{i,j}^k -\hat{\beta} _{i,j}^k \vert \hat{p}_j^k \Delta s\Delta t\\
& \leq \left[ 1- \mu _i^k\Delta t +\left(\sum_{j=1}^N \beta _{i,j}^k \Delta s \right) \Delta t \right] \vert u_i^k
\vert -\frac{\Delta t}{\Delta s} (\gamma _i^k\vert u _{i}^k \vert -\gamma _{i-1}^k \vert u _{i-1}^k \vert) \\
&\quad +\frac{\Delta t}{\Delta s} \vert \left( \gamma _{i-1}^k -\hat{\gamma}_{i-1}^k \right) \hat{p}_{i-1}^k -\left(
\gamma _i^k -\hat{\gamma} _i^k \right) \hat{p} _i^k \vert + \vert \mu _i^k -\hat{\mu}_i^k \vert \hat{p}_i^k \Delta t
+\sum_{j=1}^N \vert \beta_{i,j}^k -\hat{\beta} _{i,j}^k \vert \hat{p}_j^k \Delta s \Delta t.
\end{array}
\end{equation*}
Multiplying the above inequality by $ \Delta s$ and summing over $ i=1,2,\cdots,N$ we have\\
\begin{equation}\label{eq_sum0}
\begin{array}{ll}
\sum_{i=1}^N \vert u _i^{k+1} \vert \Delta s &\leq \sum_{i=1}^N \left[ 1-\Delta t \mu _i^k +\left( \sum_{j=1}^N
\beta_{i,j}^k \Delta s\right) \Delta t\right] \vert u_i^k \vert \Delta s \\
&\quad-\Delta t \sum _{i=1}^N\left( \gamma _i^k \vert u_i^k \vert -\gamma _{i-1}^k \vert u_{i-1}^k \vert \right)+
\Delta t \sum_{i=1}^N \vert \left( \gamma_{i-1}^k - \hat{\gamma} _{i-1}^k  \right) \hat{p}_{i-1}^k  -\left(\gamma _i^k
-\hat{\gamma}_i^k \right) \hat{p}_i^k\vert  \\
&\quad + \Delta t \sum_{i=1}^N \vert \mu _i^k -\hat{\mu}_i^k \vert \hat{p}_i^k\Delta s+\Delta t\sum_{i=1}^N
\sum_{j=1}^N \vert \beta_{i,j}^k -\hat{\beta}_{i,j}^k \vert \hat{p}_j^k \Delta s \Delta s .
\end{array}
\end{equation}
Here by assumptions (H2) and (H3)
\begin{equation}\label{eq_sum1}
\begin{array}{ll}
\sum_{i=1}^N \left[ 1- \mu _i^k\Delta t +\left( \sum_{j=1}^N \beta_{i,j}^k \Delta s\right) \Delta t\right] \vert u_i^k
\vert \Delta s
\leq \sum_{i=1}^N\left( 1+c\Delta t\right) \vert u _i^k \vert \Delta s =\left( 1+c \Delta t \right) \Vert u^k \Vert
_1.
\end{array}
\end{equation}
By assumption (H1) and the second equation of~(\ref{eq_u}) one get
\begin{equation}\label{eq_sum2}
\begin{array}{ll}
\sum _{i=1}^N\left( \gamma _i^k \vert u_i^k \vert -\gamma_{i-1}^k \vert u_{i-1}^k \vert \right)=\left(\gamma _N^k \vert
\mu _N^k \vert -\gamma _0^k \vert u_0^k \vert \right)=\gamma _0^k \vert u _0^k \vert=0.
\end{array}
\end{equation}
By assumption (H1),
\begin{equation}\label{eq_sum3}
\begin{array}{ll}
\sum_{i=1}^N \vert \left( \gamma_{i-1}^k - \hat{\gamma} _{i-1}^k  \right) \hat{p}_{i-1}^k  -\left(\gamma _i^k
-\hat{\gamma}_i^k \right) \hat{p}_i^k\vert \\
\leq \sum_{i=1}^N \vert \gamma_{i-1}^k -\hat{\gamma}_{i-1}^k \vert \vert \hat{p}_{i}^k -\hat{p}_{i-1}^k \vert
+\sum_{i=1}^N \vert \left( \gamma _i^k - \hat{\gamma}_i^k \right) -\left( \gamma_{i-1}^k \quad-\hat{\gamma}_{i-1}^k
\right) \vert \hat{p}_i^k\\
\leq \sum_{i=1}^N \vert \gamma _Q(s_{i-1},\bar{Q}^k) \vert \vert Q^k -\hat{Q}^k \vert \vert \hat{p}_{i}^k
-\hat{p}_{i-1}^k \vert
 +\sum_{i=1}^N \vert \gamma _Q(s_i,\bar{Q}^k) (Q^k-\hat{Q}^k) -\gamma _Q(s_{i-1},\bar{Q}^k) (Q^k-\hat{Q}^k) \vert
 \hat{p}_i^k \\
\leq \vert Q^k-\hat{Q}^k\vert \sup_{\mathbb{D}_3}\vert \gamma _Q \vert TV(\hat{p}^k) + \vert Q^k -\hat{Q}^k \vert
\sum_{i=1}^N \vert \gamma _Q(s_i,\bar{Q}^k)-\gamma _Q(s_{i-1},\bar{Q}^k) \vert \hat{p}_i^k \\
\leq \vert Q^k -\hat{Q}^k\vert \left[ \sup_{\mathbb{D}_3} \vert \gamma _Q \vert TV(\hat{p}^k) + c \sum_{i=1}^N
\hat{p}_i^k \Delta s \right] \\
=\vert Q^k -\hat{Q}^k \vert \left[ \sup_{\mathbb{D}_3} \vert \gamma _Q\vert TV(\hat{p}^k) + c\Vert \hat{p}^k \Vert _1
\right],\\
\end{array}
\end{equation}
where $\bar{Q}^k$ is between $Q^k$ and $\hat{Q}^k$.\\
By assumption (H2),
\begin{equation}\label{eq_sum4}
\begin{array}{ll}
\sum_{i=1}^N \vert \mu _i^k -\hat{\mu}_i^k \vert \hat{p}_i^k \Delta s
=\sum_{i=1}^N c \vert(Q^k-\hat{Q}^k)\vert \hat{p}_i^k \Delta s
\leq c \vert Q^k - \hat{Q}^k \vert \Vert \hat{p}^k \Vert _1.
\end{array}
\end{equation}
\noindent
From assumption (H3) we obtain
\begin{equation}\label{eq_sum5}
\begin{array}{ll}
\sum_{i=1}^N \sum_{j=1}^N \vert \beta_{i,j}^k -\hat{\beta}_{i,j}^k \vert \hat{p}_j^k \Delta s \Delta s \Delta t \\
\leq \Delta t \sum_{i=1}^N \sum_{j=1}^N c|Q^k-\hat{Q}^k| \hat{p}_j^k \Delta s \Delta s\\
\leq c \vert Q^k-\hat{Q}^k\vert \Delta t \sum_{i=1}^N \sum_{j=1}^N \hat{p}_j^k \Delta s \Delta s \\
\leq c \vert Q^k -\hat{Q}^k \vert \left( \sum_{j=1}^N \hat{p}_j^k \Delta s \right) \left( \sum_{i=1}^N \Delta s \right)
\Delta t\\
= c\Vert \hat{p}^k \Vert _1  \vert Q^k -\hat{Q}^k \vert \Delta t.
\end{array}
\end{equation}
\noindent
A combination of~(\ref{eq_sum0})-(\ref{eq_sum5}) and assumptions (H1)-(H3) then implies that there exists a positive
constant $\tilde{M}$ such that
$$\Vert u^{k+1} \Vert _1 \leq \left( 1+c\Delta t \right) \Vert u^k\Vert _1 + \tilde{M} \vert Q^k -\hat{Q}^k \vert\Delta
t. $$
\noindent
Note that
\begin{equation*}
\begin{array}{ll}
\vert Q^k -\hat{Q}^k \vert
=\vert \sum_{i=1}^N \left( p_i^k - \hat{p}_i^k \right) \Delta s \vert
\leq \sum_{i=1}^N \vert p_i^k -\hat{p}_i^k \vert \Delta s
\leq \sum_{i=1}^N \vert u_i^k \vert \Delta s = \Vert u^k \Vert _1
\end{array}
\end{equation*}
Therefore$$\Vert u^{k+1} \Vert _1 \leq \left( 1+ c\Delta t + \tilde{M}\Delta t \right) \Vert u^k \Vert _1 .$$
Let $\delta=c+\tilde{M}$ and we obtain the result.
\end{proof}
In the next theorem we prove that the BV solution defined in Theorem~\ref{theorem_WS1} is unique using a technique
similar to that in \cite{AI_1997}.
\begin{thm}\label{thm:uniqueness}
Suppose that $ p $ and  $ \hat{p} $ are bounded variation weak solutions of problem~(\ref{model}) corresponding to
initial conditions $\left\lbrace p^0 \right\rbrace$  and $\left\lbrace \hat{p}^0 \right\rbrace$, respectively. Then
there exists a positive constant $\rho$ such that
\begin{equation*}
\begin{array}{ll}
\Vert p(\cdot,t)-\hat{p}(\cdot,t)\Vert _1 \leq \rho \Vert p(\cdot,0)-\hat{p}(\cdot,0) \Vert _1
\end{array}
\end{equation*}
\end{thm}
\begin{proof}
Assume that $Q$ is a given Lipschitz continuous function and consider the following initial-boundary value problem:
\begin{equation}\label{system_here}
\begin{array}{ll}
\frac{\partial}{\partial t}p(s,t)+\frac{\partial}{\partial s}\left( \gamma(s,Q(t))p(s,t) \right) & \\
 \hspace{1in} =-\mu(s,Q(t))p(s,t)
+\int_{0}^{1} \beta(s,y,Q(t))p(y,t)dy,  & s\in(0,1], \; t\in(0,T], \\
\gamma(0,Q(t))p(0,t)=0, & t\in[0,T],\\
p(s,0)=p^{0}(s), &  s\in[0,1].
\end{array}
\end{equation}
Since~(\ref{system_here}) is a linear problem with local boundary conditions, it has a unique weak solution. Actually,
a weak solution can be defined as a limit of the finite difference approximation with the given numbers $Q^k=Q(t_k)$
and the uniqueness can be established by using similar techniques as in \cite{XM}. In addition, as in the proof of
Theorem~\ref{theorem_WS1}, we can show that if $ p_i^k $ and $ \hat{p}_i^k $ are solutions of the difference
scheme~(\ref{first order 1}) corresponding to given functions $ Q^k$ and $\hat{Q}^k $, respectively, then there exist
positive constants $ c_1$ and $c_2$ such that
\begin{equation}\label{ineq_cond}
\begin{array}{ll}
\Vert u^{k+1} \Vert _1 \leq \left( 1+c_1 \Delta t \right) \Vert u^k \Vert _1 + c_2 \vert Q^k -\hat{Q}^k \vert\Delta t,
\end{array}
\end{equation}
with $u^k=p^k-\hat{p}^k$.\\
The  equation~(\ref{ineq_cond}) leads to
\begin{equation*}
\begin{array}{ll}
\Vert u^k \Vert _1 &\leq (1+c_1 \Delta t)^k \Vert u^0\Vert _1 +c_2 \Delta t \sum_{r=0}^{k-1} (1+c_1 \Delta t)^r \vert
Q^{k-r-1} -\hat{Q}^{k-r-1} \vert.
\end{array}
\end{equation*}
Hence
\begin{equation}\label{ineq_u^k}
\begin{array}{ll}
\Vert u^k \Vert _1 \leq (1+c_1 \Delta t)^k \left( \Vert u^0 \Vert _1 + c_2 \Delta t \sum_{r=0}^{k-1} \vert
Q^{k-r-1}-\hat{Q}^{k-r-1} \vert  \right)
\end{array}
\end{equation}
Now from Theorem~\ref{theorem U_convergence1} one can take the limit in~(\ref{ineq_u^k}) to obtain
\begin{equation}\label{u(t)}
\begin{array}{ll}
\Vert u(\cdot,t) \Vert_1 \leq e^{c_1T} \left( \Vert u^0 \Vert _1+c_2 \int _0^t \vert Q(l)-\hat{Q}(l) \vert dl \right)
\end{array}
\end{equation}
where $u(\cdot,t)=p(\cdot,t)-\hat{p}(\cdot,t)$ and $ p(\cdot,t) $ is the unique solution of problem~(\ref{system_here})
with any set of given functions $Q(t)$ and $\hat{Q}(t)$.
We then apply the estimate given in~(\ref{u(t)}) for the corresponding solutions of~(\ref{system_here}) with two
specific functions $Q(t)$ and $\hat{Q}(t)$ which are constructed using the limits obtained in
Theorem~\ref{thm_uniqueness} as follows:
\begin{equation*}
\begin{array}{ll}
Q(t)=\int _0^1 p(s,t)ds,\hspace{20pt} \hat{Q}(t)=\int _0^1\hat{p}(s,t)ds.
\end{array}
\end{equation*}
Thus, we have
\begin{equation*}
\begin{array}{ll}
\vert Q(t)-\hat{Q}(t) \vert &= \vert \int_{0}^{1}p(s,t)ds -\int_{0}^{1} \hat{p}(s,t) ds\vert \\
&\leq \int_{0}^{1} \vert p(s,t) -\hat{p}(s,t) \vert  ds\\
&= \int_{0}^{1} \vert u(s,t) \vert ds =\Vert u(\cdot,t) \Vert _1.
\end{array}
\end{equation*}
Therefore,
 \begin{equation*}
 \begin{array}{ll}
\int _0^t \vert Q(t) -\hat{Q}(t)\vert  dl
 \leq \int _0^t \Vert u(\cdot,t) \Vert _1 dl.
 \end{array}
 \end{equation*}
Thus,
\begin{equation*}
\begin{array}{ll}
\Vert u(\cdot,t) \Vert _1 \leq e^{c_1 T} \left( \Vert u^0 \Vert _1 + c_2 \int _0^t \Vert u(\cdot,t) \Vert _1  dl
\right).
\end{array}
\end{equation*}
Using Gronwall's inequality we have
\begin{equation*}
\begin{array}{ll}
\Vert u(\cdot,t) \Vert _1 \leq e^{(c_1 T+c_2e^{c_1 T})} \Vert u^0 \Vert _1.
\end{array}
\end{equation*}
The result follows by letting $\rho=e^{(c_1 T+c_2e^{c_1 T})} $.
\end{proof}

\section{\large A second order finite difference scheme}\label{}\medskip
To achieve an accurate approximation the first order upwind scheme we discussed in the previous section would require
many grid points and thus is time consuming. In this section we develop the following second order finite difference
scheme for the DSSM based on minmod MUSCL schemes~\cite{LeVeque,SSZ}.\\
\begin{equation}\label{scheme_SOE}
\begin{array}{ll}
\frac{p_{i}^{k+1}-p_{i}^{k}}{\Delta t} +\frac{{\hat{f}_{i+\frac{1}{2}}^{k}}-\hat{f}_{i-\frac{1}{2}}^{k}}{\Delta s} =
-\mu_{i}^{k}p_{i}^{k}+{\sum_{j=0}^{N}}^{\bigstar}\beta_{i,j}^{k}p_{j}^{k}\Delta s,
& i=1,2,\cdots,N,\quad k=0,1,\cdots, L-1,\\
\gamma_{0}^{k}p_{0}^{k}=0, &  k=0,1,\cdots, L,\\
\end{array}
\end{equation}
with the initial condition $p_{i}^{0}=p^{0}(s_i)$.
Here $Q^k$ is discretized using a second order Trapezoidal rule. That is, $$Q^{k}={\sum_{i=0}^{N}}^{\bigstar}
p_{i}^{k}\Delta s=\frac{1}{2}p_0^k \Delta s+\sum_{i=1}^{N-1} p_{i}^{k} \Delta s+\frac{1}{2}p_N^k \Delta s.$$
Similarly,
$${\sum_{j=0}^{N}}^{\bigstar}\beta_{i,j}^{k}p_{j}^{k}\Delta s=\frac{1}{2}\beta_{i,0}^{k}p_{0}^{k}\Delta
s+\sum_{j=1}^{N-1}\beta_{i,j}^{k}p_{j}^{k}\Delta s +\frac{1}{2}\beta_{i,N}^{k}p_{N}^{k}\Delta s.$$\\
\noindent
The finite difference scheme~(\ref{scheme_SOE}) can be rewritten as
\begin{equation}
p_{i}^{k+1} = p_{i}^{k}-\frac{\Delta t}{\Delta s}(\hat{f}_{i+\frac{1}{2}}^{k}-\hat{f}_{i-\frac{1}{2}}^{k})-\mu_{i}^{k}
p_{i}^{k} \Delta t+ \left({\sum_{j=0}^{N}}^{\bigstar} \beta_{i,j}^{k} p_{j}^{k} \Delta s\right) \Delta t, \quad
i=1,2,\cdots,N.
\label{eql_p_i^k+1}
\end{equation}
Here the limiter is defined as
\begin{equation}\label{eq:fhat}
 \hat{f}_{i+\frac{1}{2}}^{k} =
\left\{
\begin{array}{ll}
\gamma_{i}^{k} p_{i}^{k}+\frac{1}{2}(\gamma_{i+1}^{k}-\gamma_{i}^{k})p_i^k+\frac{1}{2}\gamma_{i}^{k}
mm(\Delta_{+}p_{i}^{k},\Delta_{-} p_{i}^{k}), & i=2,\cdots, N-2,
\\ \gamma_{i}^{k}p_{i}^{k}, & i=0,1,N-1,N.
\end{array}
\right.
\end{equation}
The minmod function $mm$ is defined by
\begin{equation*}
\begin{array}{ll}
 mm(a,b) = \frac{sign(a)+sign(b)}{2} \min(\vert a \vert,\vert b \vert).
 \end{array}
 \end{equation*}
Therefore,
\begin{equation*}
\begin{array}{ll}
 0\leqslant \frac{mm(a,b)}{a} \leqslant 1 \hspace{20pt} \text{and} \hspace{20pt} 0\leqslant \frac{mm(a,b)}{b}\leqslant
 1, \hspace{20 pt} \forall a,b\neq 0.
 \end{array}
 \end{equation*}
As in~\cite{SSZ} we define $B_i^k$ and $D_i^k$ by
\begin{equation*}
\begin{array}{ll}
B_{i}^{k} =
\left\{
\begin{array}{ll}
\dfrac{1}{2}\left(\gamma_{i+1}^{k}+\gamma_{i}^{k}+\gamma_{i}^{k}\dfrac{\emph{mm}(\Delta_{+}p_i^k,
\Delta_{-}p_i^k)}{\Delta_{-}p_i^k}-\gamma_{i-1}^{k}
\dfrac{\emph{mm}(\Delta_{-}p_i^k, \Delta_{-}p_{i-1}^k)}{\Delta_{-}p_i^k}\right), & i=3,\dots ,N-2,\\
\dfrac{1}{2}\left(\gamma_{i+1}^{k}+\gamma_{i}^{k}+\gamma_{i}^{k}\dfrac{\emph{mm}(\Delta_{+}p_i^k,\Delta_{-}p_{i}^k)}{\Delta_{-}p_i^k}\right),&
i=2,\\
\dfrac{1}{2}\left(2\gamma_{i}^{k}-\gamma_{i-1}^{k}\dfrac{\emph{mm}(\Delta_{-}p_i^k,\Delta_{-}p_{i-1}^k)}{\Delta_{-}p_i^k}\right),
& i=N-1,\\
\gamma_{i}^{k}, & i=1,N,
\end{array}
\right.
\end{array}
\end{equation*}
\begin{equation*}
\begin{array}{ll}
D_{i}^{k} =
\left\{
\begin{array}{ll}
\dfrac{1}{2}\left(\Delta_{+}\gamma_{i}^{k}+\Delta_{-}\gamma_{i}^{k}\right), & i=3,\dots ,N-2,\\
\dfrac{1}{2} \Delta_{+} \gamma_{i}^{k} +\Delta_{-} \gamma_{i}^{k}, & i=2,\\
\dfrac{1}{2} \Delta_{-} \gamma_{i}^{k}, & i=N-1,\\
\Delta_{-}\gamma_{i}^{k},& i=1,N.
\end{array}
\right.
\end{array}
\end{equation*}
Note that
\begin{equation*}
\begin{array}{ll}
2(B_{i}^{k}-D_{i}^{k})=
\left\{
\begin{array}{ll}
\gamma_{i}^{k}\left(1+\frac{mm(\Delta_{+} p_{i}^{k},\Delta_{-} p_{i}^{k})}{\Delta_{-} p_{i}^{k}}
\right)+\gamma_{i-1}^{k}\left(1-\frac{mm(\Delta_{-}p_{i}^{k},\Delta_{-}p_{i-1}^{k})}{\Delta_{-}p_{i}^{k}}\right),&
i=3,\cdots ,N-2,\\
2\gamma_{i-1}^{k}+\gamma_{i}^{k}\frac{mm(\Delta_{+} p_{i}^{k},\Delta_{-}p_{i}^{k})}{\Delta_{-}p_{i}^{k}} ,& i=2 , \\
\gamma_{i}^{k}+\gamma_{i-1}^{k}\left(1-\frac{mm(\Delta_{-}p_{i}^{k},\Delta_{-}p_{i-1}^{k})}{\Delta_{-}p_{i}^{k}}\right),
& i= N-1,\\
2\gamma_{i-1}^{k}, & i=1, N.
\end{array}
\right.
\end{array}
\end{equation*}
One can easily see from assumption (H1) that
\begin{equation}\label{BD}
\begin{array}{cc}
\vert B_{i}^{k}\vert \leqslant \frac{3}{2} \sup_{ \mathbb{D}_1} \vert \gamma\vert\leqslant \frac{3}{2} c, \quad\quad
B_{i}^{k} - D_{i}^{k}\geq 0.
\end{array}
\end{equation}
The finite difference scheme~(\ref{scheme_SOE}) can then be written in a more compact way as follows:
\begin{equation}
\begin{array}{ll}
p_{i}^{k+1} = \left(1-\frac{\Delta t}{\Delta s}B_{i}^{k}-\mu_{i}^{k}\Delta t\right)p_i^k+\frac{\Delta t}{\Delta
s}(B_{i}^{k}-D_{i}^{k})p_{i-1}^{k}+\left({\sum_{j=0}^{N}}^{\bigstar}\beta_{i,j}^{k}p_{j}^{k}
\Delta s\right)\Delta t, & \text{for}\quad i=1,2,\cdots, N.
\end{array}
\label{equ:scheme_compact}
\end{equation}

\subsection{\large Estimates of the finite difference scheme}\label{}\medskip
From a biological point of view, it is very important that our scheme preserves non-negativity of solutions. We will
first show this property in the following lemma.
\begin{lem} \label{lem_L1_bnd_0}
The finite difference scheme~(\ref{scheme_SOE}) has a unique nonnegative solution.
\end{lem}
\begin{proof}
 From assumption (H4) we have $p_{i}^{0}\geq 0$ for $i=0,1,\cdots, N$. Also, by the second equation
 in~(\ref{scheme_SOE}) and assumption (H1), $p_{0}^{k}=0$ for $k\geq 0$.
Moreover, by assumptions (H1)-(H3) and (H5), one observes that
\begin{equation}
      1-\frac{\Delta t}{\Delta s}B_{i}^{k}-\mu_{i}^{k}\Delta t \geq 1-\frac{\Delta t}{\Delta s}\frac{3}{2}
      \sup_{\mathbb D_1} \vert\gamma\vert-\sup_{\mathbb D_{1}}\vert \mu\vert \Delta t
     \geq 1- \frac{\Delta t}{\Delta s} \frac{3}{2}c-c\Delta t\geq 0.
\end{equation}
Therefore, by induction it follows that $ p_{i}^{k}\geq 0$ for $i=1,2,\cdots, N$, $k\geq 1$, and thus the system has a
unique nonnegative solution.
\end{proof}
The next lemma shows that the numerical approximations are  bounded in $\ell^{1}$ norm.
\begin{lem} \label{lem_L1_bnd_1}
For some positive constant $M_{5}$, the following estimate holds.
\begin{equation}
\Vert p^{k} \Vert _{1} \leq \exp(c T) \Vert p^{0} \Vert _{1}\equiv M_{5} , \quad\quad \text{for}\quad k=0,1,\cdots, L.
\end{equation}
\end{lem}
\begin{proof}
Multiplying the first equation in~(\ref{eql_p_i^k+1}) by $\Delta s$
and summing over $i=1,2,\cdots,N,$ we have
\begin{equation*}
\begin{array}{ll}
\Vert p^{k+1}\Vert_{1}&=\sum_{i=1}^{N}p_{i}^{k} \Delta s
-\sum_{i=1}^{N}(\hat{f}_{i+\frac{1}{2}}^{k}-\hat{f}_{i-\frac{1}{2}}^{k}) \Delta t -\sum_{i=1}^{N}\mu_{i}^{k}p_{i}^{k}
\Delta t \Delta s +\sum_{i=1}^{N}\left({\sum_{j=0}^{N}}^{\bigstar}\beta_{i,j}^{k}p_{j}^{k}\Delta s\right)\Delta s\Delta
t \\
&=\Vert p^{k}\Vert _{1}-(\gamma_{N}^{k}p_{N}^{k}-\gamma_{0}^{k}p_{0}^{k})\Delta t-
\sum_{i=1}^{N}\mu_{i}^{k}p_{i}^{k}\Delta s\Delta
t+\sum_{i=1}^{N}\left({\sum_{j=0}^{N}}^{\bigstar}\beta_{i,j}^{k}p_{j}^{k}\Delta s\right)\Delta s\Delta t.
\end{array}
\end{equation*}
Therefore, by assumptions (H1)-(H3) one can see that
\begin{equation*}
\begin{array}{ll}
\Vert p^{k+1}\Vert_{1}&\leq \Vert p^{k}\Vert _{1}+\sum_{i=1}^{N}({\sum_{j=0}^{N}}^{\bigstar}\beta_{i,j}^{k}\Delta s
p_{j}^{k})\Delta s\Delta t\\
&\leq \Vert p^{k}\Vert_{1}+c \sum_{i=1}^{N}\Vert p^{k}\Vert_{1}\Delta s\Delta t \\
&\leq(1+c\Delta t)\Vert p^{k}\Vert_{1},
\end{array}
\end{equation*}
which implies the estimate.
\end{proof}
\noindent
Note that
\begin{eqnarray*}
Q^{k}={\sum _{i=0}^{N}}^\bigstar p_{i}^{k}\Delta s=\sum _{i=1}^{N} p_i^k\Delta s-\frac{1}{2}p_N^k\Delta s\leqslant
\sum_{i=1}^{N}p_{i}^{k}\Delta s=\Vert p^{k}\Vert_{1}\leq M_{5}.
\end{eqnarray*}
We now define $\mathbb{D}_4=[0,1]\times [0, M_5]$.

The following lemma establishes $l^\infty$ bounds of the numerical approximations.
\begin{lem} \label{lem_L1_bnd_2}
There exists a positive constant $M_{6}$ such that$$\Vert p^{k}\Vert_{\infty}\leqslant M_{6},\quad\quad\text{for}\quad
k=0,1,\cdots,L.$$
\end{lem}
\begin{proof}
If $\Vert p^{k}\Vert_{\infty}$ is obtained at the left boundary then $\Vert p^{k}\Vert_{\infty}= p_{0}^{k}=0$ for
$k\geq 0$. Otherwise, assume that $p_{i}^{k+1}=\Vert p^{k+1} \Vert_{\infty}$, for some $1\leqslant i \leqslant N$.
From equation~(\ref{equ:scheme_compact}), assumptions (H1)-(H3) and (H5) we have
\begin{eqnarray*}
\|p^{k+1}\|_{\infty} &\leqslant &\left(1-\frac{\Delta t}{\Delta s}B_{i}^{k}-\mu_{i}^{k}\Delta t\right)\Vert
p^{k}\Vert_{\infty}+\frac{\Delta t}{\Delta s}(B_{i}^{k}-D_{i}^{k})\Vert p^{k}\Vert _{\infty}
+\left({\sum_{j=0}^{N}}^{\bigstar}\beta_{i,j}^{k}\Delta s\right)\Vert p^{k}\Vert_{\infty} \Delta t\\
&\leqslant &(1+c\Delta t) \Vert p^{k}\Vert_{\infty}-\frac{\Delta t}{\Delta s} D_{i}^{k}\Vert p^{k}\Vert_{\infty}.
\end{eqnarray*}
By assumption (H1), $|\gamma_{i}^{k}-\gamma_{i-1}^{k}|=|\gamma_{s}(\hat{s}_{i},Q^{k})|\Delta s \leqslant c\Delta s$ and
thus $-D_i^k\leq \frac{3}{2}c \Delta s$.\\
Therefore,
 $$\Vert p^{k+1}\Vert _{\infty}\leqslant (1+\frac{5}{2}c\Delta t)^{k+1}\Vert p^{0}\Vert_{\infty}.$$
 The result then follows easily from the above inequality.
\end{proof}
In the next lemma we show that the approximations $p_{i}^{k}$ are of bounded total variation.
\begin{lem}\label{lem_L1_bnd_3}
There exists a constant $M_{7}$ such that $$TV(p^{k})\leqslant M_7, \quad\quad\text{for}\quad k=0,1,\cdots,L.$$
\end{lem}
\begin{proof}
From~(\ref{equ:scheme_compact}) we have
\begin{equation*}
\begin{array}{ll}
p_{i+1}^{k+1}-p_{i}^{k+1}&=
\left(1-\frac{\Delta t}{\Delta s}B_{i+1}^{k}\right)(p_{i+1}^{k}-p_{i}^{k})+\frac{\Delta t}{\Delta
s}(B_{i}^{k}-D_{i}^{k})(p_{i}^{k}-p_{i-1}^{k})-\frac{\Delta t}{\Delta s}(D_{i+1}^{k}-D_i^k)p_i^k\\
&\quad-\Delta t(\mu_{i+1}^{k} p_{i+1}^{k}-\mu _i^k p_i^k)+
\left({\sum_{j=0}^{N}}^{\bigstar}\beta_{i+1,j}^{k}p_{j}^{k}\Delta s
\right) \Delta t-\left({\sum_{j=0}^{N}}^{\bigstar}\beta_{i,j}^{k}p_{j}^{k}\Delta s
\right) \Delta t,\\
\end{array}
\end{equation*}
for $i=1,\cdots,N-1$.\\
Therefore,
\begin{equation}\label{eq:TV_SOE}
\begin{array}{ll}
TV(p^{k+1})&= \vert p_{1}^{k+1} - p_0^{k+1}\vert +\sum_{i=1}^{N-1}\vert p_{i+1}^{k+1}-p_i^{k+1}\vert\\
&\leqslant \vert p_1^{k+1}-p_0^{k+1}\vert +\sum_{i=1}^{N-1} \vert \left(1-\frac{\Delta t}{\Delta
s}B_{i+1}^{k}\right)(p_{i+1}^{k}-p_i^k)+\frac{\Delta t}{\Delta s}(B_i^k -D_i^k)(p_i^k-p_{i-1}^k)\vert\\
&\quad +\sum_{i=1}^{N-1}\vert D_{i+1}^k-D_i^k\vert p_i^k\frac{\Delta t}{\Delta s}+\sum_{i=1}^{N-1}\vert \mu_{i+1}^k
p_{i+1}^k-\mu _i^k p_i^k\vert \Delta t\\
&\quad +\sum_{i=1}^{N-1}\vert {\sum_{j=0}^{N}}^{\bigstar}\beta_{i+1,j}^{k}p_{j}^{k}\Delta
s-{\sum_{j=0}^{N}}^{\bigstar}\beta_{i,j}^{k}p_{j}^{k} \Delta s\vert \Delta t\\
&= \vert p_{1}^{k+1} - p_0^{k+1}\vert +I_1 +I_2+I_3+I_4.
\end{array}
\end{equation}
We now estimate the bound of $TV(p^k)$ term by term.
\begin{equation}
\begin{array}{ll}
\vert p_{1}^{k+1} - p_0^{k+1}\vert
&=\left(1-\frac{\Delta t}{\Delta s}B_1^k-\mu _1^k \Delta t\right)p_1^k+\frac{\Delta t}{\Delta
s}(B_1^k-D_1^k)p_0^k+\left(
{\sum_{j=0}^{N}}^{\bigstar}\beta_{1,j}^{k}p_{j}^{k} \Delta s\right) \Delta t\\
&=\left(1-\frac{\Delta t}{\Delta s}\gamma _1^k-\mu _1^k\Delta t
\right)p_1^k+\left({\sum_{j=0}^{N}}^{\bigstar}\beta_{1,j}^{k}p_{j}^{k} \Delta s\right) \Delta t.
\end{array}
\end{equation}
By assumptions (H1) and (H5),
\begin{equation}
\begin{array}{ll}
I_1
&\leqslant \sum_{i=1}^{N-1}\left(1-\frac{\Delta t}{\Delta s}B_{i+1}^{k}\right)\vert p_{i+1}^{k}-p_i^k\vert
+\frac{\Delta t}{\Delta s}(B_i^k -D_i^k)\vert p_i^k-p_{i-1}^k\vert\\
&\leqslant \sum_{i=1}^{N-1} \vert p_{i+1}^k -p_i^k\vert -\frac{\Delta t}{\Delta s} \sum_{i=1}^{N-1}
\left(B_{i+1}^k\vert p_{i+1}^{k}-p_i^k\vert -B_i^k\vert p_i^k-p_{i-1}^k\vert \right)-\frac{\Delta t}{\Delta
s}\sum_{i=1}^{N-1} D_i^k \vert p_i^k-p_{i-1}^k\vert\\
&\leqslant \sum_{i=1}^{N-1} \vert p_{i+1}^k -p_i^k\vert -\frac{\Delta t}{\Delta s}\left(B_N^k\vert p_N^k-p_{N-1}^k\vert
-B_1^k\vert p_1^k-p_0^k\vert \right)- \frac{\Delta t}{\Delta s}\sum_{i=1}^{N-1} D_i^k \vert p_i^k-p_{i-1}^k \vert\\
&\leqslant \sum_{i=1}^{N-1} \vert p_{i+1}^k -p_i^k\vert +\frac{\Delta t}{\Delta s}\gamma _1^k p_1^k +\frac{\Delta
t}{\Delta s}\sum_{i=1}^{N-1} \frac{3}{2} \sup_{\mathbb{D}_4}\vert \gamma _i^k-\gamma _{i-1}^k\vert\hspace{5pt} \vert
p_i^k-p_{i-1}^k\vert\\
&\leqslant \sum_{i=1}^{N-1} \vert p_{i+1}^k -p_i^k\vert+\frac{\Delta t}{\Delta s}\gamma _1^k p_1^k +\frac{3}{2}c
TV(p^k)\Delta t.
\end{array}
\end{equation}
By assumption (H1),
\begin{equation}\label{eq_I2}
\begin{array}{ll}
I_2
&= \sum_{i=3}^{N-3}\vert D_{i+1}^k-D_i^k\vert p_i^k\frac{\Delta t}{\Delta s}+\sum_{i=1,2,N-2,N-1} \vert D_{i+1}^k
-D_i^k\vert p_i^k \frac{\Delta t}{\Delta s}\\
&\leqslant \sum_{i=3}^{N-3} \vert D_{i+1}^k-D_i^k\vert p_i^k \frac{\Delta t}{\Delta s} +8 \sup_{\mathbb{D}_4} \vert
D_i^k\vert p_i^k \frac{\Delta t}{\Delta s}\\
&\leqslant \sum_{i=3}^{N-3} \vert D_{i+1}^k-D_i^k\vert p_i^k\frac{\Delta t}{\Delta s}+12 c \parallel p^k\parallel
_\infty \Delta t.
\end{array}
\end{equation}
From assumption $ (H1)$ we have
\begin{equation}\label{eq_D_D}
\begin{array}{ll}
\vert D_{i+1}^k-D_i^k\vert &= \frac{1}{2}\vert \left(\bigtriangleup _+\gamma _{i+1}^k+\bigtriangleup _-\gamma
_{i+1}^k\right) -\left(\bigtriangleup _+\gamma _i^k +\bigtriangleup _-\gamma _i^k\right) \vert \\
&=\frac{1}{2} \vert \left( \gamma _{i+2}^k -\gamma _{i+1}^k\right) -\left(\gamma _i^k-\gamma _{i-1}^k\right) \vert\\
&=\frac{1}{2}\vert \gamma _s(\hat{s}_{i+2},Q^k)\Delta s-\gamma _s(\hat{s}_i,Q^k)\Delta s\vert\\
&\leq  \frac{1}{2} c\vert \hat{s}_{i+2}-\hat{s}_i\vert \Delta s=c(\Delta s)^2,
\end{array}
\end{equation}
where $ \hat{s}_i \in [s_{i-1},s_i]$ and $\hat{s}_{i+2} \in [s_{i+1},s_{i+2}]$ for $i=3,4,\ldots,N-3$. \\
Therefore by combining~(\ref{eq_I2}) and~(\ref{eq_D_D}) we obtain
\begin{equation}
\begin{array}{ll}
I_2 \leqslant \sum_{i=3}^{N-3} c(\Delta s)^2 p_i^k\frac{\Delta t}{\Delta s}+12c\| p^k\| _\infty \Delta t
\leqslant c\| p^k\| _1 \Delta t+12 c\| p^k\| _\infty \Delta t.
\end{array}
\end{equation}
We have from assumption (H2) that
\begin{equation}
\begin{array}{ll}
I_3
& =  \sum_{i=1}^{N-1} \left| (\mu _{i+1}^k -\mu _i^k)p_{i+1}^k+\mu _i^k (p_{i+1}^k-p_i^k) \right| \Delta t\\
& \leq  c\Delta s \sum_{i=1}^{N-1} p_{i+1}^k \Delta t+\sup_{\mathbb{D}_4}\mu \sum_{i=1}^{N-1}\vert p_{i+1}^k-p_i^k\vert
\Delta t\\
& \leq  c \| p^k\| _1 \Delta t +c TV(p^k)\Delta t.
\end{array}
\end{equation}
By assumption (H3),
\begin{equation}\label{eq_I4}
\begin{array}{ll}
I_4 \leq{\sum_{j=0}^{N}}^{\bigstar}\left(\sum_{i=1}^{N-1}|\beta_{i+1,j}^{k}-\beta_{i,j}^{k}|\right)p_{j}^{k}\Delta
s\Delta t
\leq c{\sum_{j=0}^{N}}^{\bigstar}p_{j}^{k}\Delta s\Delta t\leq c\|p^{k}\|_{1}\Delta t.
\end{array}
\end{equation}
A combination of~(\ref{eq:TV_SOE})-(\ref{eq_I4}) then leads to
\begin{equation}
\begin{array}{ll}
TV(p^{k+1})&= \left( 1-\frac{\Delta t}{\Delta s} \gamma _1^k-\mu _1^k \Delta t\right) p_1^k +\left(
{\sum_{j=0}^{N}}^\bigstar \beta_{1,j}^kp_j^k\Delta s\right) \Delta t +
\sum_{i=1}^{N-1}|p_{i+1}^{k}-p_{i}^{k}|+\frac{\Delta t}{\Delta s}\gamma _1^k p_1^k\\
& \quad+\frac{3}{2} c \Delta t TV(p^k)
+ c\| p^k\| _1 \Delta t+12 c | p^k\| _\infty \Delta t
+ c \| p^k\| _1 \Delta t + c TV(p^k)\Delta t+ c\| p^k\| _1 \Delta t\\
&\leq  \left(1+c_1 \Delta t\right) TV(p^k)+c_2 \Delta t,
\end{array}
\end{equation}
where $c_{1}=\frac{5}{2}c$ and $c_{2}=4c M_5+12cM_6$.
The result then follows.
\end{proof}
Next we will show that the finite difference approximations are $\ell_{1}$ lipschitz continuous in $t$.
\begin{lem} \label{lem_L1_bnd_5}
There exists a positive constant $M_{8}$ such that for any $m>n>0$ the following estimates hold:
\begin{equation*}
\sum_{i=1}^{N} \left| \frac{p_i^m-p_i^n}{\Delta t}\right| \Delta s\leq M_8 (m-n).
\end{equation*}
\end{lem}
\begin{proof}
From~(\ref{equ:scheme_compact}) and assumptions (H1)-(H3), we have
\begin{equation}
\begin{array}{ll}
\sum_{i=1}^{N}\left| \frac{p_i^{k+1}-p_i^{k}}{\Delta t}\right| \Delta s
&= \sum_{i=1}^{N}\vert -B_i^k p_i^k-\mu _i^k p_i^k \Delta s+B_i^kp_{i-1}^k-D_i^kp_{i-1}^k+{\sum _{j=0}^N}^{\bigstar}
\beta _{i,j}^{k} p_j^k\Delta s \Delta s \vert\\
&\leq  \frac{3}{2}\sup_{\mathbb{D}_4} \gamma\sum_{i=1}^{N}\vert p_i^k-p_{i-1}^k\vert + \sup_{\mathbb{D}_4} \mu \|
p^k\|_1 +\sum_{i=1}^{N}\frac{3}{2}\sup_{\mathbb{D}_4}|\gamma_{i+1}^k-\gamma_i^k|p_{i-1}^k\\
&\quad + \sum_{i=1}^{N}{\sum _{j=0}^N}^{\bigstar}\sup_{\mathbb{D}_4}|\beta| p_j^k\Delta s \Delta s\\
&\leq  \frac{3}{2} c TV(p^k) +c\| p^k\| _1 +\frac{3}{2}c\| p^k\| _1+c\| p^k\| _1.
\end{array}
\end{equation}
Thus by Lemmas~\ref{lem_L1_bnd_1} and~\ref{lem_L1_bnd_3} there exists a positive constant $M_8$ such that
 $$\sum_{i=1}^{N}\left|\frac{p_i^{k+1}-p_i^{k}}{\Delta t}\right| \Delta s\leq M_8.$$
 Therefore
\begin{eqnarray*}
\sum_{i=1}^{N} \left| \frac{p_i^m-p_i^n}{\Delta t}\right| \Delta s\leq\sum_{i=1}^N\sum_{k=n}^{m-1}\left|
\frac{p_i^{k+1}-p_i^{k}}{\Delta t}\right|\Delta s\leq M_8 (m-n).
\end{eqnarray*}
\end{proof}
\subsection{\large Convergence of the difference approximations}\label{}\medskip
 We again follow similar notation as in \cite{J_Smoller} and define a set of functions $\lbrace P_{\Delta s,\Delta
 t}\rbrace$ by $ \lbrace P_{\Delta s,\Delta t}(s,t)\}=p_i^k$ for $s\in [s_{i-1},s_i), t\in [t_{k-1},t_k)$,
 $i=1,2,\cdots,N$, and $k=1,2,\cdots,L$. Then by the Lemmas~\ref{lem_L1_bnd_1} - ~\ref{lem_L1_bnd_5}, the set of
 functions $\lbrace P_{\Delta s,\Delta t}\rbrace$ is compact in the topology of $\mathcal{L}^1((0,1)\times (0,T))$.
 Hence following the proof of Lemma $16.7$ on page $276$ in~\cite{J_Smoller} we obtain the following result.
\begin{thm}\label{theorem U_convergence}
There exists a subsequence of functions $\lbrace P_{\Delta s_r,\Delta t_r}\rbrace \subset \lbrace P_{\Delta s,\Delta
t}\rbrace$ which converges to a function $p \in BV\left( [0,1] \times [0,T] \right) $ in the sense that for all $t>0$,
$$\int _0^1 \vert P_{\Delta s_r,\Delta t_r}-p(s,t)\vert ds\longrightarrow 0,$$
$$\int _0^T \int_0^1 \vert P_{\Delta s_r,\Delta t_r}-p(s,t)\vert ds dt \longrightarrow 0$$
as $r \rightarrow \infty $ (i.e., $\Delta a_r, \Delta s _r, \Delta t_r \rightarrow 0$). Furthermore, there exist
constants $M_{9}$ depending on $\Vert p^0 \Vert_{BV\left( [0,1]\times [0,T] \right)}$ such that the limit function
satisfies
$$ \| p \|_{BV\left( [0,1] \times [0,T] \right)}\leq M_{9}.$$
\end{thm}

We show in the next theorem that the limit function $p(s,t)$ constructed by the finite difference scheme is a weak
solution to problem~(\ref{model}).
\begin{thm}
The limit function $p(s,t)$ defined in Theorem~\ref{theorem U_convergence} is a weak solution of the DSSM. Moreover, it
satisfies
$$\Vert p \Vert _{L^\infty\left( (0,1) \times (0,T) \right)}\leq \exp\left(\frac{5}{2}cT\right) \Vert p^{0} \Vert
_\infty.$$
\end{thm}
\begin{proof}
Let $\phi \in C^1\left( [0,1] \times [0,T] \right)$ and denote the value of $ \phi (s_i,t_k)$ by $\phi_i^k$.
\noindent
Multiplying equation~(\ref{eql_p_i^k+1}) by $\phi _i^{k+1}$ and rearranging some terms we have
\begin{equation}\label{p_iphi_i}
\begin{array}{ll}
p_i^{k+1}\phi _i^{k+1} -p_i^k \phi_i^{k} &=p_i^k (\phi_ i^{k+1}-\phi_ i^k)+\frac{\Delta t}{\Delta
s}[\hat{f}_{i-\frac{1}{2}}^k(\phi_i^{k+1}-\phi_{i-1}^{k+1}) +(\hat{f}_{i-\frac{1}{2}}^k \phi_
{i-1}^{k+1}-\hat{f}_{i+\frac{1}{2}}^k \phi _i^{k+1})]\\
&\quad - \mu _i^k p_i^k \phi_ i^{k+1}\Delta t + {\sum_{j=0}^N}^{\bigstar} \beta _{i,j}^k p_j^k  \phi_i^{k+1}\Delta s
\Delta t.
\end{array}
\end{equation}
\noindent
Multiplying the above equation by $\Delta s$, summing over $ i=1,2,\cdots,N$, $k=0, 1,\cdots, L-1$, and applying
$p_0^k=0$ and $\gamma_N^k=0$ we obtain,
\begin{equation}\label{weak_solution_sum}
\begin{array}{ll}
\sum_{i=1}^N\left( p_i^L \phi_ i^L- p_i^0\phi_ i^0\right) \Delta s &= \sum_{k=0}^{L-1} \sum_{i=1}^N p_i^k \frac{\phi
_i^{k+1}-\phi_ i^k}{\Delta t} \Delta s \Delta t\\
&\quad+\sum_{k=0}^{L-1} \sum_{i=0}^{N-1}\hat{f}_{i+\frac{1}{2}}^k\frac{\phi _{i+1}^{k+1}-\phi_ {i}^{k+1}}{\Delta
s}\Delta s \Delta t-\sum_{k=0}^{L-1} \sum_{i=1}^N \mu_i^k p_i^k \phi _i^{k+1}\Delta s \Delta t\\
&\quad+\sum_ {k=1}^{L-1} \sum_{i=1}^{N} {\sum_{j=1}^N}^{\bigstar} \beta _{i,j}^k p_j^k \phi_ i^{k+1} \Delta s \Delta t
\Delta s.
\end{array}
\end{equation}\label{sump_iphi_i}
\noindent
Note that by~(\ref{eq:fhat}) one have
\begin{equation}
\begin{array}{ll}
\sum_{k=0}^{L-1} \sum_{i=0}^{N-1}\hat{f}_{i+\frac{1}{2}}^k\frac{\phi _{i+1}^{k+1}-\phi_ {i}^{k+1}}{\Delta s}\Delta s
\Delta t\\
=\sum_{k=0}^{L-1} [\gamma_{0}^{k}p_{0}^k+\gamma_{1}^{k}p_{1}^k+\gamma_{N-1}^{k}p_{N-1}^k+\sum_{i\in
{J_1}}\frac{\gamma_{i}^k+\gamma_{i+1}^k}{2}p_i^k
+\sum_{i\in J_2}\frac{\gamma_{i+1}^kp_i^k+\gamma_{i}^kp_{i+1}^k}{2}\\
\quad+\sum_{i\in J_3}\frac{\gamma_{i+1}^k p_i^k+2\gamma_{i}^kp_{i}^k-\gamma_{i}^k p_{i-1}^k}{2}]\frac{\phi
_{i+1}^{k+1}-\phi_ {i}^{k+1}}{\Delta s}\Delta s \Delta t,
\end{array}
\end{equation}
where $J_1=\{2\leq i\leq N-2: \text{sign}(\Delta_{+}p_{i}^{k})\text{sign}(\Delta_{-} p_{i}^{k})=-1, \;\text{or}\;
\text{sign}(\Delta_{+}p_{i}^{k})\text{sign}(\Delta_{-} p_{i}^{k})=0\}$, $J_2=\{2\leq i\leq N-2: \Delta_{-} p_{i}^k\geq
\Delta_{+} p_{i}^k>0, \; \text{or}\; \Delta_{-} p_{i}^k\leq\Delta_{+} p_{i}^k < 0\}$, $J_3=\{2\leq i\leq N-2:
\Delta_{+} p_{i}^k>\Delta_{-} p_{i}^k>0,\; \text{or}\; \Delta_{+} p_{i}^k<\Delta_{-} p_{i}^k< 0\}$. One can easily
check that
$J_1\cup J_2\cup J_3=\{2,3,\cdots,N-3,N-2\}$.
Now we could rewrite (\ref{weak_solution_sum}) as
\begin{equation}
\begin{array}{ll}
\sum_{i=1}^N\left( p_i^L \phi_ i^L- p_i^0\phi_ i^0\right) \Delta s 
&= \sum_{k=0}^{L-1} \sum_{i=1}^N p_i^k \frac{\phi_i^{k+1}-\phi_ i^k}{\Delta t} \Delta s \Delta t
+ \sum_{k=0}^{L-1} [\gamma_{0}^{k}p_{0}^k+\gamma_{1}^{k}p_{1}^k+\gamma_{N-1}^{k}p_{N-1}^k\\
&\quad +\sum_{i\in{J_1}}\frac{\gamma_{i}^k+\gamma_{i+1}^k}{2}p_i^k
+\sum_{i\in J_2}\frac{\gamma_{i+1}^kp_i^k+\gamma_{i}^kp_{i+1}^k}{2}\\
&\quad+\sum_{i\in J_3}\frac{\gamma_{i+1}^k p_i^k+2\gamma_{i}^kp_{i}^k-\gamma_{i}^k p_{i-1}^k}{2}]\frac{\phi _{i+1}^{k+1}-\phi_ {i}^{k+1}}{\Delta s}\Delta s \Delta t\\
&\quad-\sum_{k=0}^{L-1} \sum_{i=1}^N \mu_i^k p_i^k \phi _i^{k+1}\Delta s \Delta t
+\sum_ {k=1}^{L-1} \sum_{i=1}^{N} {\sum_{j=1}^N}^{\bigstar} \beta _{i,j}^k p_j^k \phi_ i^{k+1} \Delta s \Delta t\Delta s.
\end{array}
\end{equation}
Since $p_i^k$ is piecewise constant and $\phi$ is smooth, and the integrals are limits of step functions, we have
\begin{equation}\label{convergence_final}
\begin{array}{ll}
\int_0^1 P_{\Delta s,\Delta t} (s,t) \phi(s,t) ds + \delta_1 - \int_{0}^1 P_{\Delta s,\Delta t}(s,0) \phi(s,0)ds+ \delta_2\\
= \int_0^t\int_0^1 P_{\Delta s,\Delta t}(s,\tau) \phi _\tau(s,\tau) ds d\tau +\delta_3
+ \int_0^t \left \{\int_0^{\Delta s}\gamma(s,Q(\tau)) P_{\Delta s,\Delta t}(s,\tau)\phi_s(s,\tau) ds\right .\\
\quad+\int_{1-\Delta s}^1 \gamma(s,Q(\tau)) P_{\Delta s,\Delta t}(s,\tau)\phi_s(s,\tau) ds
\left .
+\int_{J_1} \gamma(s,Q(\tau)) P_{\Delta s,\Delta t}(s,\tau)\phi_s(s,\tau) ds\right . \\
\left .
\quad+ \int_{J_2} \gamma(s,Q(\tau)) P_{\Delta s,\Delta t}(s,\tau)\phi_s(s,\tau) ds
 +\int_{J_3} \gamma(s,Q(\tau)) P_{\Delta s,\Delta t}(s,\tau)\phi_s(s,\tau) ds \right \} d\tau+\delta_4\\
 \quad-\int_0^t \int_0^1 P_{\Delta s,\Delta t}(s,\tau) \mu(s,Q(\tau)) \phi(s,\tau) ds d\tau +\delta_5 \\
\quad +\int_0^t \int_0^1 \phi(s,\tau)  \int_0^1 P_{\Delta s,\Delta t}(s,\tau)  \beta(s,y,Q(\tau))dy ds d\tau + \delta_6.
\end{array}
\end{equation}
$\delta_i \rightarrow 0$, $i=1,2,\cdots,6$, as $\Delta s, \Delta t \rightarrow 0$ and by the choice of the initial values $\int_{0}^1 P_{\Delta s,\Delta t}(s,0) \phi(s,0)ds\rightarrow \int_0^1 p^0(s) \phi ds$ as $\Delta s \rightarrow 0$. By Theorem~\ref{theorem U_convergence} $\int_0^1 | P_{\Delta s, \Delta t} -p(s,t)|ds \rightarrow 0$ and $\int_0^t \int_0^1 |P_{\Delta s,\Delta t}-p(s,t)| ds dt \rightarrow 0$ as $\Delta s, \Delta t \rightarrow 0 $. Combining the above fact and~(\ref{convergence_final}) and following a similar argument used in the proof of Lemma (16.9) on page 280 of~\cite{J_Smoller}, we can show that the limit of the difference approximations in Theorem~\ref{theorem U_convergence} is a weak solution to problem~(\ref{model}). The bound on $\Vert p \Vert _{L^\infty}$ is obtained by taking the limit in the bounds of the difference approximation in Lemma~\ref{lem_L1_bnd_2}.
\end{proof}

\section{Weak* connection between CSSM and DSSM}
The aim of this section is to establish a relationship between solutions of the single state-at-birth  model CSSM and the
distributed states-at-birth model DSSM. In particular we show that if the
distribution of new recruits in the DSSM becomes concentrated at the left-boundary ($s=0$), then solutions of the DSSM
converge to solutions of the CSSM in the weak* topology. To this end, we have the following theorem

\begin{thm}\label{weak*}
Let $\{\beta_n(s,y,Q)\}_{n\geq 1}$ be a sequence of reproductive functions of DSSM. Assume
$\beta_n(s,y,Q)=\beta_{1,n}(s)\beta_2(y,Q)$ such that
\begin{itemize}
\item[(A)] $\beta_2\in C^{1}([0,1]\times [0,\infty))$ and $ 0\leq \beta_2(y,Q)\leq c $.
\item[(B)] $\beta_{1,n}\in C^{1}([0,1])$ and $\int_0^1 \beta_{1,n}(s)ds =1$ for each $n\geq 1$.
\item[(C)] For every test function $\xi \in C[0,1] $,
$ \int_0^1 \beta_{1,n}(s)\xi(s)ds \to \xi(0)$, as $n \to \infty$.
\end{itemize}
Then the weak solution $p_n$ of DSSM (\ref{model}) corresponding to $\beta_n$ converges to the weak solution, $\hat p$,
of CSSM \eqref{model_McKendrick} in the weak* topology, i.e., as $ n \to \infty$,  $\int_0^1 p_{n}(s, t)\eta(s)ds \to
\int_0^1 \hat{p}(s, t)\eta(s) ds$ for every $\eta\in C[0,1]$.
\end{thm}
\begin{proof}
It can be seen that $\beta_n$ satisfies assumption (H3). Thus for each $\beta_n$, there exists a weak solution $p_n$ of
DSSM~(\ref{model}) which satisfies equation~(\ref{weak solution of DSSM}). We denote the total population by $Q_n(t) =
\int_0^1 p_n(s,t) ds$.
Now, let  $\phi\equiv 1$ in~(\ref{weak solution of DSSM}) and apply property (B) we get
\begin{equation}\label{weak solution of DSSM_phi1}
\begin{array}{ll}
\int _0^1 p_n(s,t)ds-\int _0^1p^0(s)ds\\
=-\int _0^t\int _0^1  p_n(s,\tau)\mu(s,Q_n(\tau))dsd\tau
+ \int _0^t\int _0^1\beta_2(y,Q_n(\tau))p_n(y,\tau)dyd\tau.
\end{array}
\end{equation}
Since  $\mu\geq 0$ and $p_n\geq 0$, By (A) one have
\begin{equation}\label{}
\| p_n(\cdot,t) \|_1 \leq \| p^0\|_1+ c \int_0^t \|p_n(\cdot,\tau)\|_1 d\tau.
\end{equation}
Using Gronwall's inequality we have
\begin{equation}\label{equ:bound}
\| p_n(\cdot,t) \|_1 \leq \exp(ct)\| p^0\|_1\leq \exp(cT)\| p^0\|_1.
\end{equation}
Combining~(\ref{equ:bound}) and assumption (H4) one can easily see that the solutions $p_n$ of DSSM are bounded in
$L^1$ norm uniformly in $n$.
Thus, there exists a subsequence $\{p_{n_{i}}\}$ of $\{p_n\}$ that converges in the weak* topology to $\hat{p}$ as
$n_{i}\to \infty$.
More specifically, for every $\eta \in C[0,1]$, $\int_0^1 p_{n_{i}}(s, t)\eta(s)ds \to \int_0^1 \hat{p}(s, t)\eta(s)
ds$ as $n_{i}\to \infty$.
Letting $\eta\equiv 1$, we get $Q_{n_i}\to \hat{Q}$ as $n_{i}\to \infty$. Since by assumption (H1), $ \gamma(s,Q)$ is
continuously differentiable with respect to $s$ and $Q$, $\gamma(s, Q_{n_i}) \to \gamma(s, \hat{Q})$, as $n_{i}\to
\infty$.
Similarly, applying assumptions (H2) and (A) one gets $\mu(s, Q_{n_i})\to \mu(s, \hat{Q})$ and $\beta_2(s, Q_{n_i})\to
\beta_2(s, \hat{Q})$ as $n_{i}\to \infty$.\\
Now letting $n_i\to \infty$ in equation~(\ref{weak solution of DSSM}) and applying (C) we obtain
\begin{equation}\label{}
\begin{array}{ll}
\int _0^1 \hat{p}(s,t)\phi(s,t)ds-\int _0^1 p^0(s)\phi(s,0)ds\\
=\int _0^t\int _0^1  \hat{p}(s,\tau)[\phi _\tau(s,\tau) +\gamma (s,\hat{Q}(\tau))\phi
_s(s,\tau)-\mu(s,\hat{Q}(\tau))\phi(s,\tau)]dsd\tau \\
\quad + \int _0^t \phi(0,\tau) \int _0^1 \beta_2(y,\hat{Q}(\tau))\hat{p}(y,\tau)dyd\tau,
\end{array}
\end{equation}
for any $\phi \in C^1([0,1] \times [0,T])$. Therefore $\hat{p}$ satisfies equation (1.2) in~\cite{AI_1997} and thus is
a weak solution of the CSSM with initial condition $p^0(s)$ and reproduction function $\beta_2(s,Q)$. Since the weak
solution of CSSM is unique \cite{AI_1997} we get that $p_n \to \hat p$ the unique weak solution of CSSM.
\end{proof}

\section{Numerical simulations and examples}\label{}

In this section we present several numerical simulations to demonstrate the performance of the first order explicit
upwind scheme~(\ref{first order 1}) and the second order explicit scheme~(\ref{scheme_SOE}) developed in the previous
sections. To better demonstrate their capability in solving the DSSM we compare the schemes with another second-order
explicit upwind method~\cite{Patankar}, which is given by:\\
\begin{equation}\label{second order upwind}
\begin{array}{ll}
\frac{p_i^{k+1}-p_i^k}{\Delta t}+\frac{\gamma _i^k p_i^k}{\Delta s} = -\mu _i^kp_i^k+{\sum_{j=0}^{N}}^{\bigstar}
\beta_{i,j}^{k}p_{j}^{k} \Delta s,&   i=1,\\
\frac{p_i^{k+1}-p_i^k}{\Delta t}+\frac{3 \gamma _i^k p_i^k -4\gamma_{i-1}^{k}p_{i-1}^{k}}{\Delta s} = -\mu
_i^kp_i^k+{\sum_{j=0}^{N}}^{\bigstar} \beta_{i,j}^{k}p_{j}^{k} \Delta s, &  i=2,\\
\frac{p_i^{k+1}-p_i^k}{\Delta t}+\frac{3 \gamma _i^k p_i^k -4\gamma_{i-1}^{k}p_{i-1}^{k}+\gamma_{i-2}^k
p_{i-2}^k}{\Delta s} = -\mu _i^kp_i^k+{\sum_{j=0}^{N}}^{\bigstar}\beta_{i,j}^{k}p_{j}^{k} \Delta s, &  3\leq i\leq
N,\quad 0 \leq k \leq L-1,\\
\gamma_ 0^k p_0^k=0, &  0 \leq k \leq L,\\
p_i^0=p^0(s_i), &  0 \leq i \leq N.\\
\end{array}
\end{equation}
Here $Q^k$ is discretized by the same Trapezoidal rule as used in scheme~(\ref{scheme_SOE}).
We also utilize scheme~(\ref{scheme_SOE}) to investigate the connection between the two population models: DSSM and
CSSM. At last we apply the numerical scheme~(\ref{scheme_SOE}) to show supercritical Hopf-bifurcation in a distributed
states-at-birth model.
Throughout this section we use uniformly spaced grid points for both $s$ and $t$. For simplicity of presentation, we
denote the first order explicit upwind scheme~(\ref{first order 1}), the second order explicit
scheme~(\ref{scheme_SOE}) and the second order explicit upwind scheme~(\ref{second order upwind}) by FOEU, SOEM and
SOEU, respectively.

\subsection{Validation of the numerical methods against an exact solution}
This example is merely designed to test the order of accuracy of the schemes for smooth solutions and thus may be
biologically irrelevant. To this end we choose the parameter values such that the resulting model is nonlinear and
would yield an exact solution. Let the initial condition be $p^0(s)=s$. The rest of the parameter values are chosen to
be the following:
\begin{equation*}
\begin{array}{ll}
T=8.0,\\
\beta(s,y,Q)=1+4sQ,\\
\gamma(s,Q)=(1-s)/2,\\
\mu(s,Q)=2Q.
\end{array}
\end{equation*}
With this choice of model ingredients it can be easily verified that $p(s,t)=se^{t}$ is an exact solution of the DSSM.
Given the exact solution, we can show numerically the order of accuracy of the schemes by means of an error table. We
ran seven simulations for each scheme with step sizes being halved with each successive simulation. Then we calculated
the corresponding $L^1$ norm of the error in each simulation for all schemes. In the initial simulation we let $N=10$
and $L = 40$. Based on these consecutive $L_1$ errors we calculated the orders of accuracy, and listed the results in
Table 1. This table indicates that the designed order of accuracy is obtained by all three schemes for this smooth
solution of the model.

To have a better understanding of the order of accuracy, we plot the logarithmic value of the $L_1$ norm of the errors
for all three schemes in Figure 1. Combining Table 1 and Figure 1, one can see clearly that the two second-order
methods SOEU and SOEM perform almost equally well in this case when the model parameters and solutions are smooth
functions. Also it seems that to get a similar accuracy as obtained in the second-order methods, the first-order method
requires to adopt step sizes at least $32$ times smaller.

\begin{table}\label{table_error}
\caption{$L^1$ errors and orders of accuracy for FOEU, SOEU and SOEM schemes.}
\vspace{.1 in}
\centering
\begin{tabular}[c]{|c| c| c | c | c | c | c | c |}
\hline
\multirow{2}{*}{$N$} & \multirow{2}{*}{$L$}
& \multicolumn{2}{c|}{FOEU} & \multicolumn{2}{c|}{SOEU} & \multicolumn{2}{c|}{SOEM} \\ \cline{3-8}
         &      & $L^1$ error & Order & $L^1$ error & Order  &$L^1$ error  & Order    \\ \hline
   10  & 40   & 2.51E-01 &      & 3.68E-03  & & 6.30E-03 &          \\ \hline
   20  & 80   & 1.15E-01 & 1.12 & 9.63E-04  & 1.94 & 1.66E-03 & 1.92  \\ \hline
   40  & 160  & 5.56E-02 & 1.05 & 2.50E-04  & 1.95 & 4.33E-04 & 1.94 \\ \hline
   80  & 320  & 2.74E-02 & 1.02 & 6.39E-05  & 1.97 & 1.11E-04 & 1.97  \\ \hline
   160 & 640  & 1.36E-02 & 1.01 & 1.62E-05  & 1.98 & 2.81E-05 & 1.98  \\ \hline
   320 & 1280 & 6.78E-03 & 1.00 & 4.07E-06  & 1.99 & 7.07E-05 & 1.99   \\ \hline
   640 & 2560 & 3.39E-03 & 1.00 & 1.02E-06  & 2.00 & 1.77E-06 & 1.99   \\ \hline
\end{tabular}
\end{table}

\begin{figure}[ht]\label{figure:L1norm_error}
\centering
    \includegraphics[height=4.25in, width=5.0in]{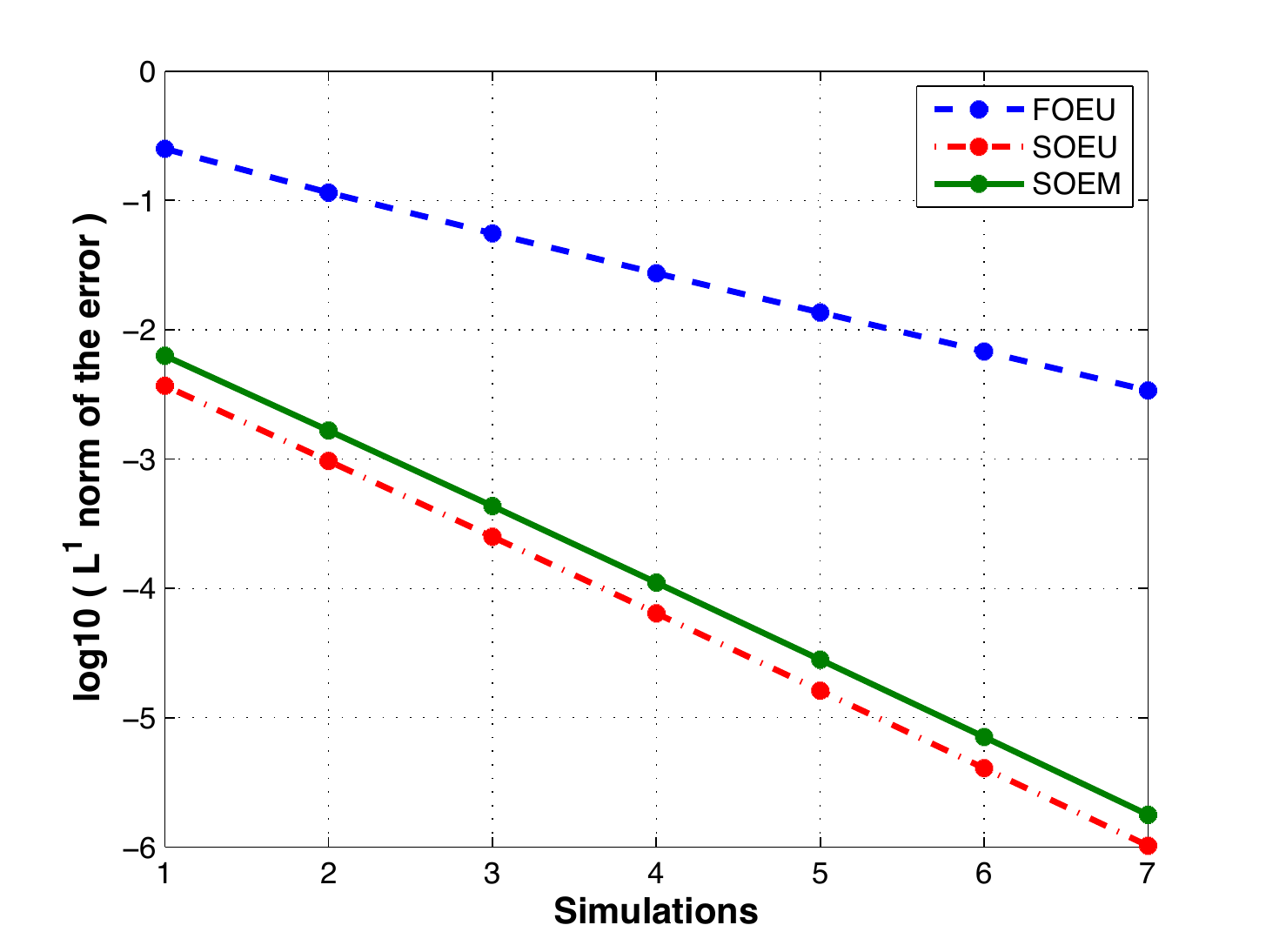}
\caption {\footnotesize The logarithmic value of $L^1$ norm of the errors for FOEU, SOEU and SOEM schemes in seven
simulations.}
\end{figure}

\subsection{Behavior at discontinuity}
The superiority of the SOEM scheme over both FOEU and SOEU methods is clear once solutions become discontinuous. To
show this, we set the initial condition in the DSSM as
\begin{equation*}
\begin{array}{ll}
p^0(s)=
\left\{\begin{array}{lll}
0.5, & 0 \leq s < 0.25,\\
1, & 0.25 \leq s \leq 0.75,\\
0.5, & 0.75 < s \leq 1,\\
\end{array}\right.\\
\end{array}
\end{equation*}
and choose the following model ingredients
\begin{equation*}
\begin{array}{ll}
\beta(s,y,Q)=
\left\{\begin{array}{lll}
0, & s \leq y-\frac{1}{2m},\\
m, & y-\frac{1}{2m} \leq s \leq y+\frac{1}{2m},\\
0, & s>y+\frac{1}{2m},
\end{array}\right.\\
\gamma(s,Q)=(1-s)/2,\\
\mu(s, Q)=2\exp(0.1Q),\\
\end{array}
\end{equation*}
with $m$ being a positive constant.\\
The above parameter choices introduce several discontinuities in the solution: two that arise from the initial
condition and another that arises
from the incompatibility of the boundary and initial condition at the origin. In the numerical simulations we use
$T=1.0$, $N=400$, and $L=800$. The results corresponding to different values of $m$ for all three methods are shown in
Figure 2. One can observe that SOEM scheme performs better than both the FOEU and SOEU schemes. It demonstrates
sharper accuracy in capturing the discontinuity in the solution than both upwind schemes without generating (decaying)
spurious oscillations.

\begin{figure}[htb]\label{discontinuous}
\centering
\includegraphics[height=2.1 in, width=3.2in]{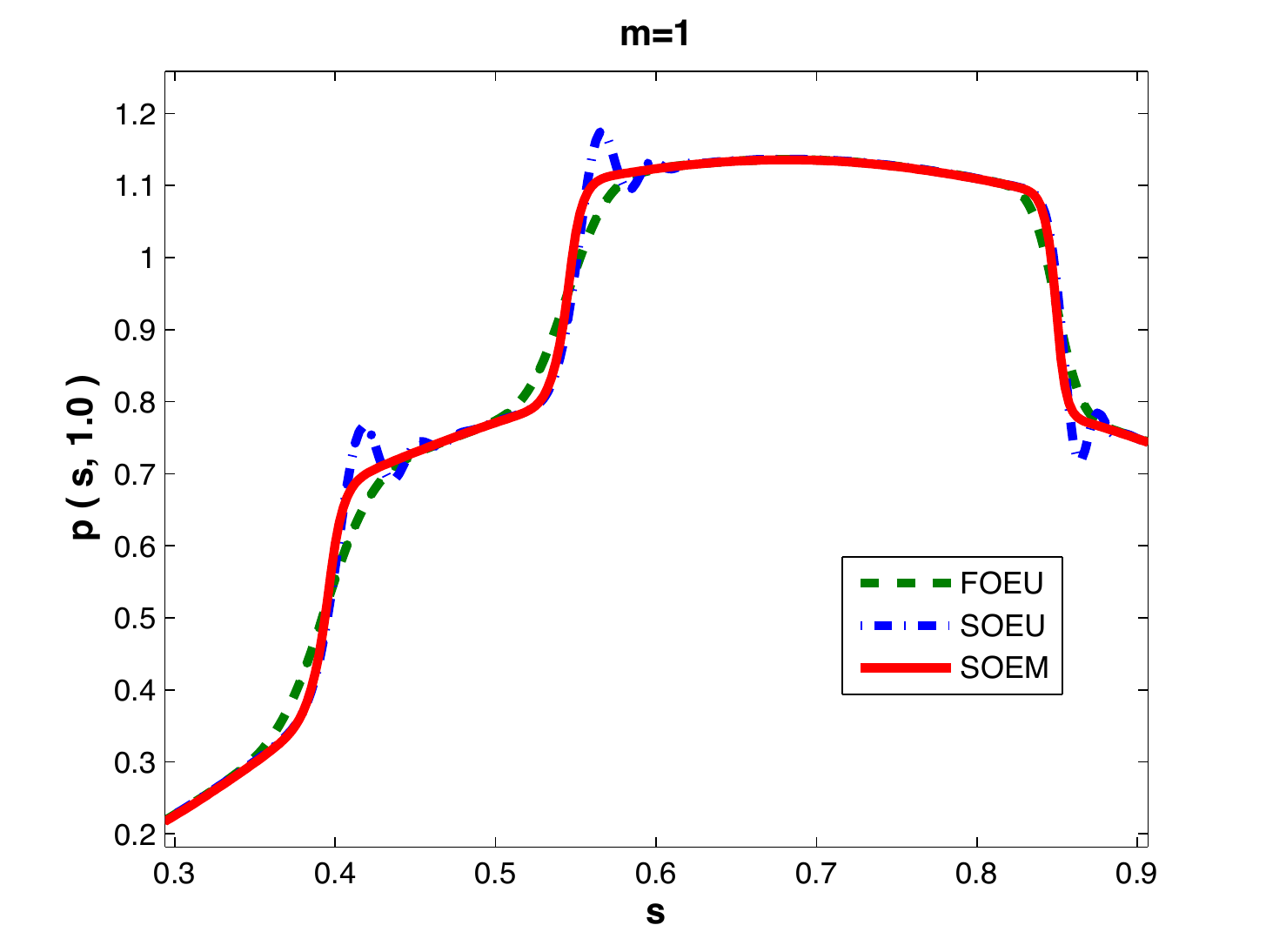}\hspace{-0.1in}
\includegraphics[height=2.1 in, width=3.2in]{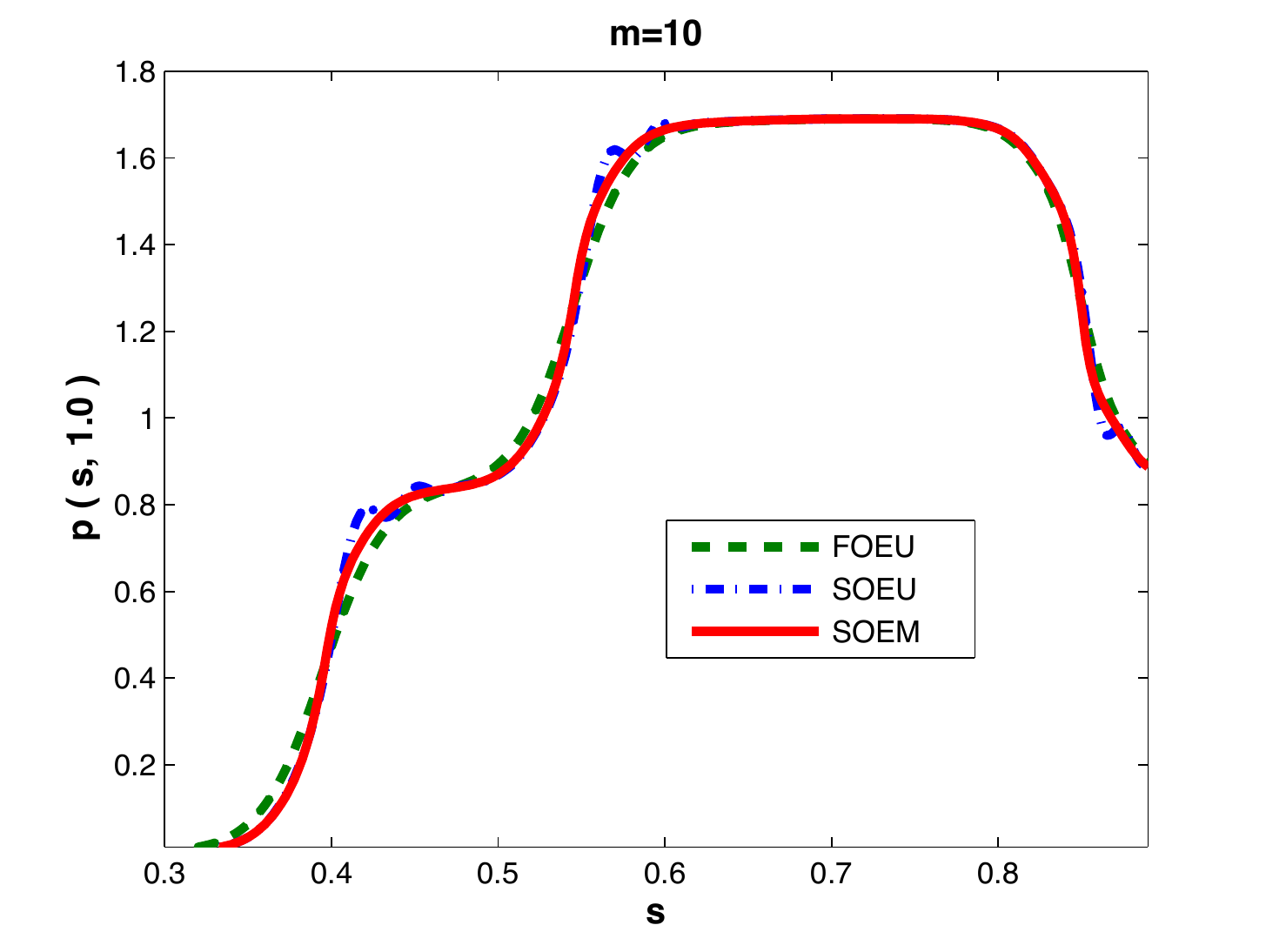}
\includegraphics[height=2.1 in, width=3.2in]{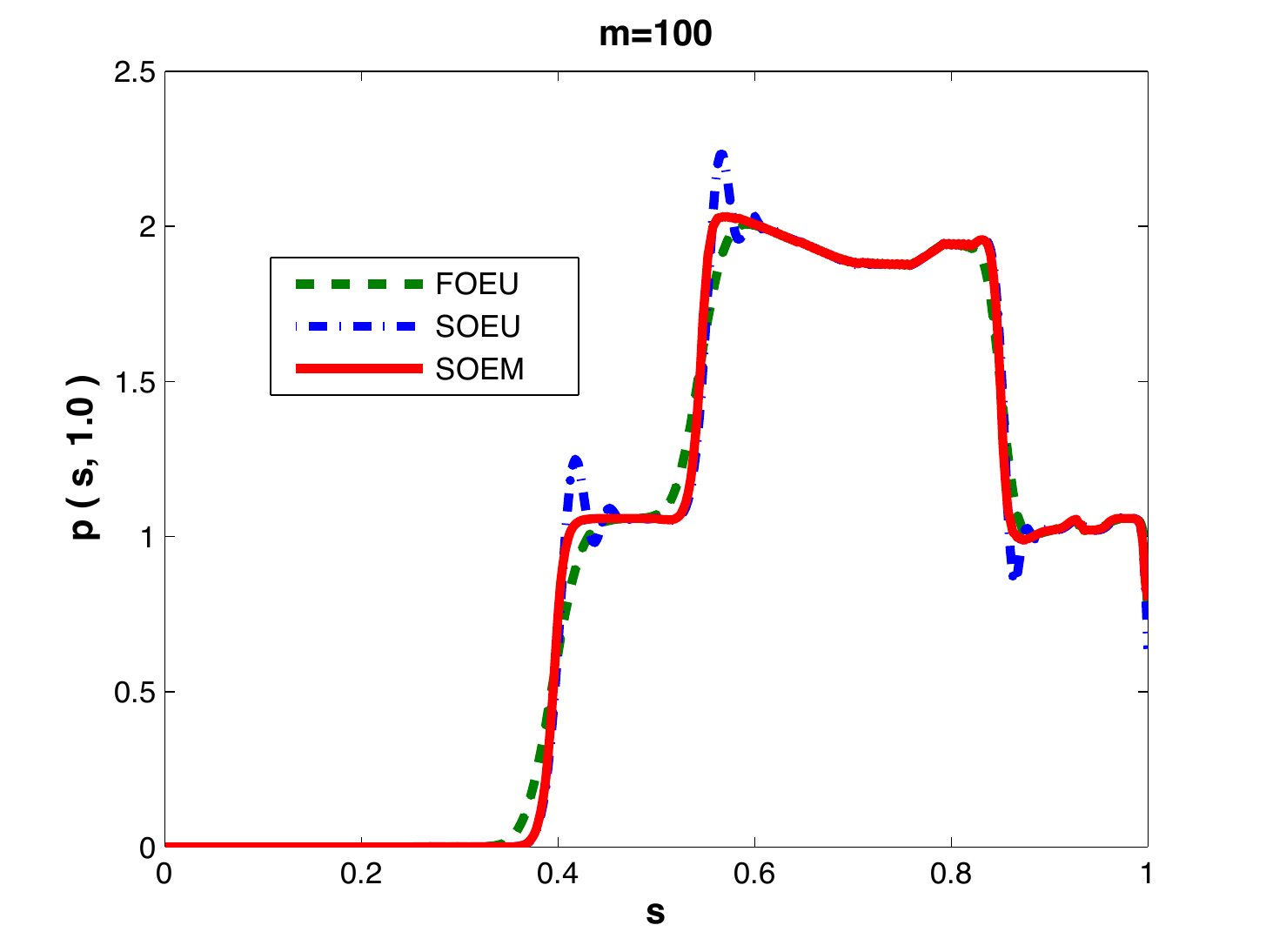}\hspace{-0.1in}
\includegraphics[height=2.1 in, width=3.2in]{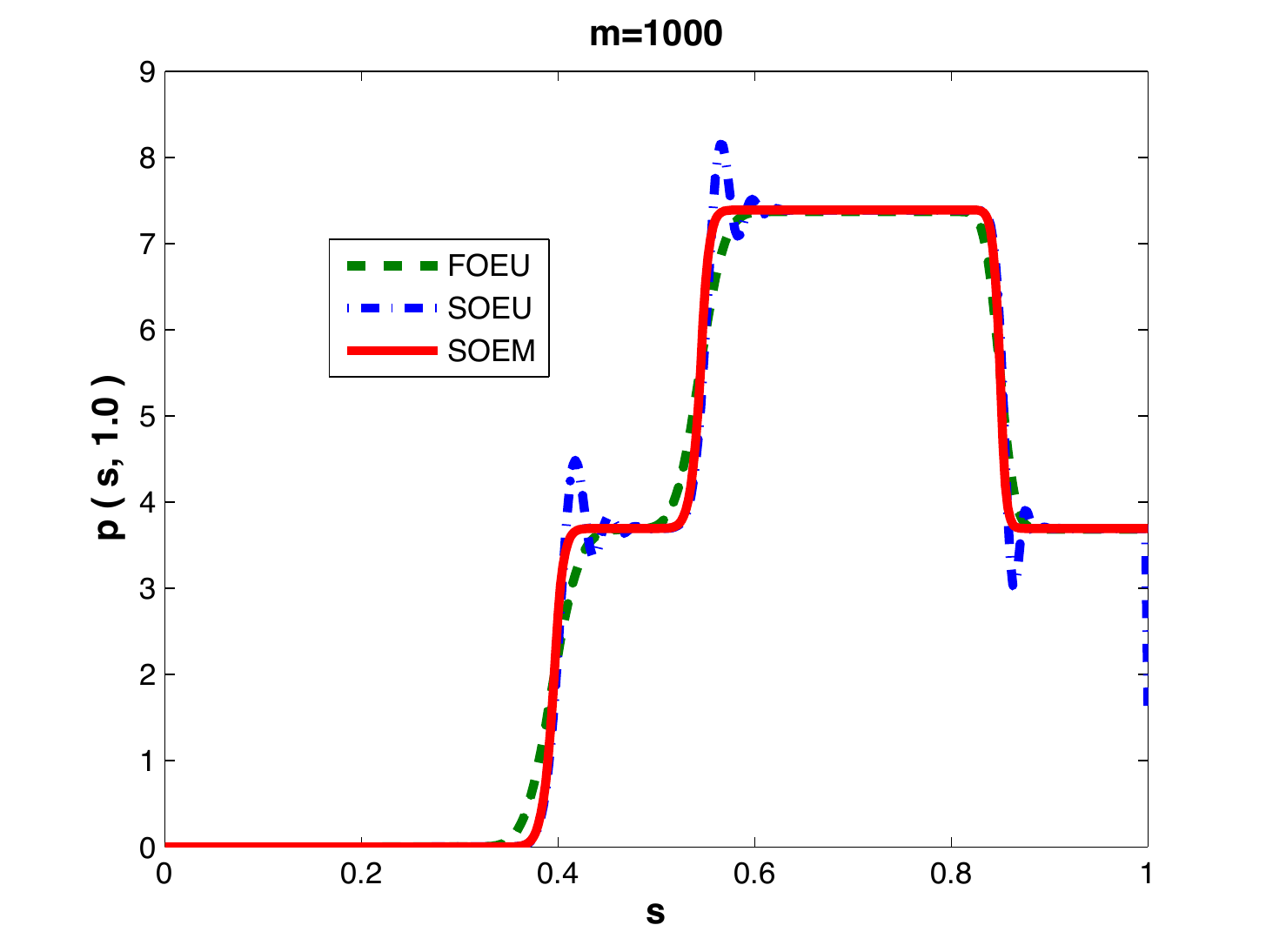}
\caption {\footnotesize The size distributions at time $t=1$ are plotted for all three schemes for $m =1, 10, 100$, and
$1000$. }
\end{figure}

\subsection{Numerical verification of the convergence of solutions of DSSM to CSSM}
In this section we provide some numerical corroboration to Theorem~\ref{weak*}. To this end, we set the initial
condition to be $p^{0}(s)=s^3$ and the parameters involved in the DSSM as follows:
\begin{equation*}
\begin{array}{ll}
\gamma(s, Q)= \frac{1}{2} (1-s),\\
\mu(s, Q) =1,
 \end{array}
\end{equation*}
To invoke Theorem \ref{weak*} let
\begin{equation*}
\beta_1(s,a,b)=\frac{s^{a-1}(1-s)^{b-1}}{B(a,b)},\quad s\in [0,1],
\end{equation*}
be the Beta probability density function with parameters $a$ and $b$; while
\begin{equation*}
B(a,b)=\int_{0}^{1}x^{a-1}(1-x)^{b-1}\ dx
\end{equation*}
is the Beta function.

Note that if we fix $a>1$ and choose a sequence $b_n \to \infty$ then from the properties of the Beta probability
density function the sequence of fertility functions of DSSM $\beta_n(s,y,Q)=\beta_1(s; a,b_n) \beta_2 (y,Q)$ with
$\beta_2\equiv 1$ satisfies the conditions in Theorem \ref{weak*}. Thus, this theorem states that the solutions of the
DSSM will converge to the solution of CSSM with fertility $\beta_2=1$ in the weak* topology. The numerical results we
present below demonstrate that this convergence may actually hold in a stronger topology, namely $L^1$.

In the numerical simulations presented below, we choose $a=1.01$ and $b=50, 75, 100$, respectively.  The graphs of the
fertility function $\beta$ corresponding to these values of $b$ are shown in Figure \ref{beta_function}. To simulate
the CSSM, we let the fertility $\beta_2=1$ and for all other parameters we use the same values given above for DSSM. We
then apply SOEM for solving the DSSM and CSSM. The results of the densities of DSSM and CSSM at $T=0.8$ are presented
in Figure \ref{beta_distribution}.

\begin{figure}[H]
\centering
\includegraphics[height=3.5in, width=5.0in]{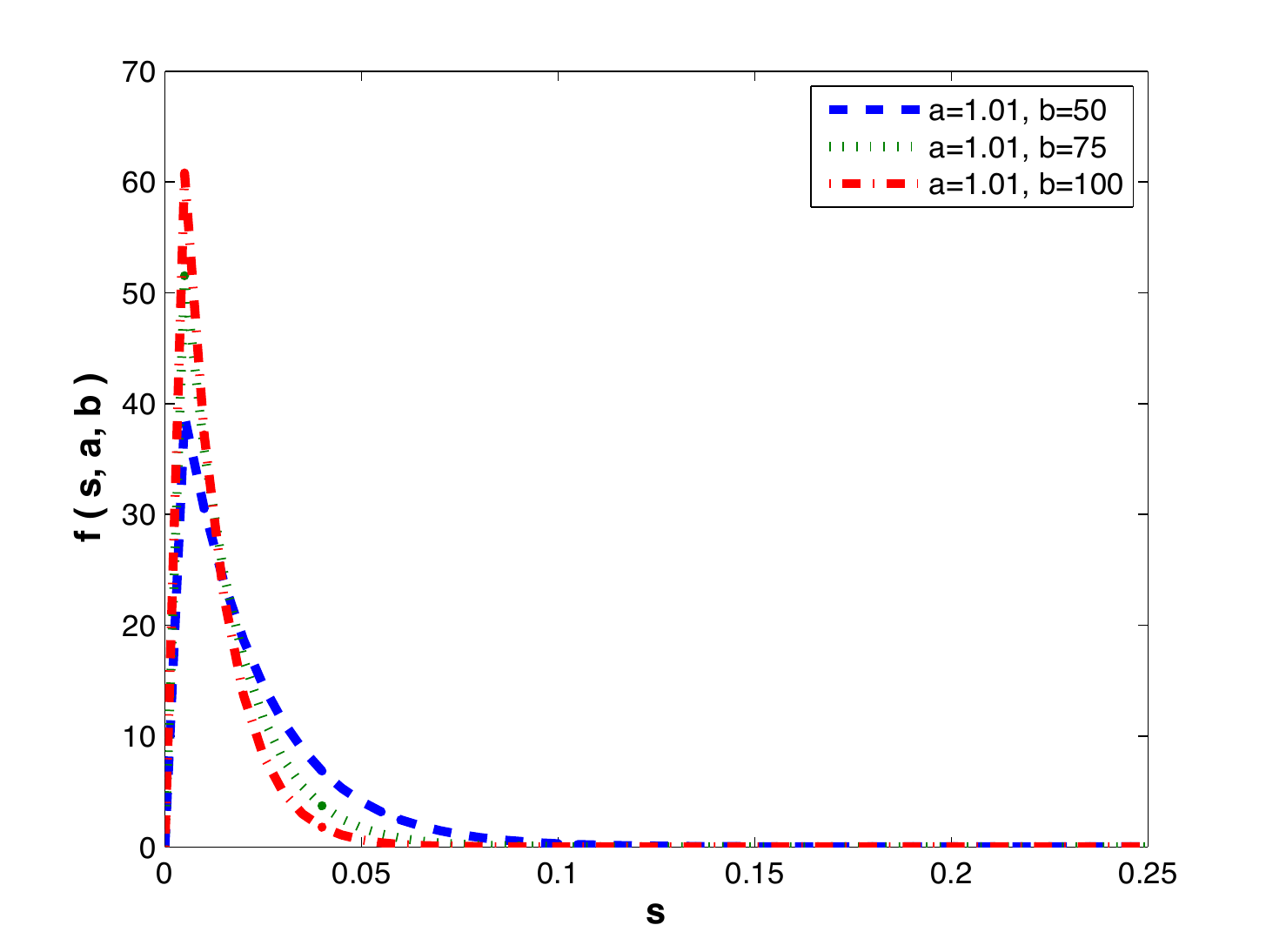}
\caption {\footnotesize The graph of the distributed recruitment rate $\beta_1(s,a,b)$ for $a=1.01$ and $b=50,75, 100$.
}
\label{beta_function}\end{figure}
In Figure 4 we plot the solution $p$ at $t=0.8$  of the DSSM for different values of $b$, and the corresponding
solution of the CSSM.
\begin{figure}[H]
\centering
\includegraphics[height=3.5in, width=5.0in]{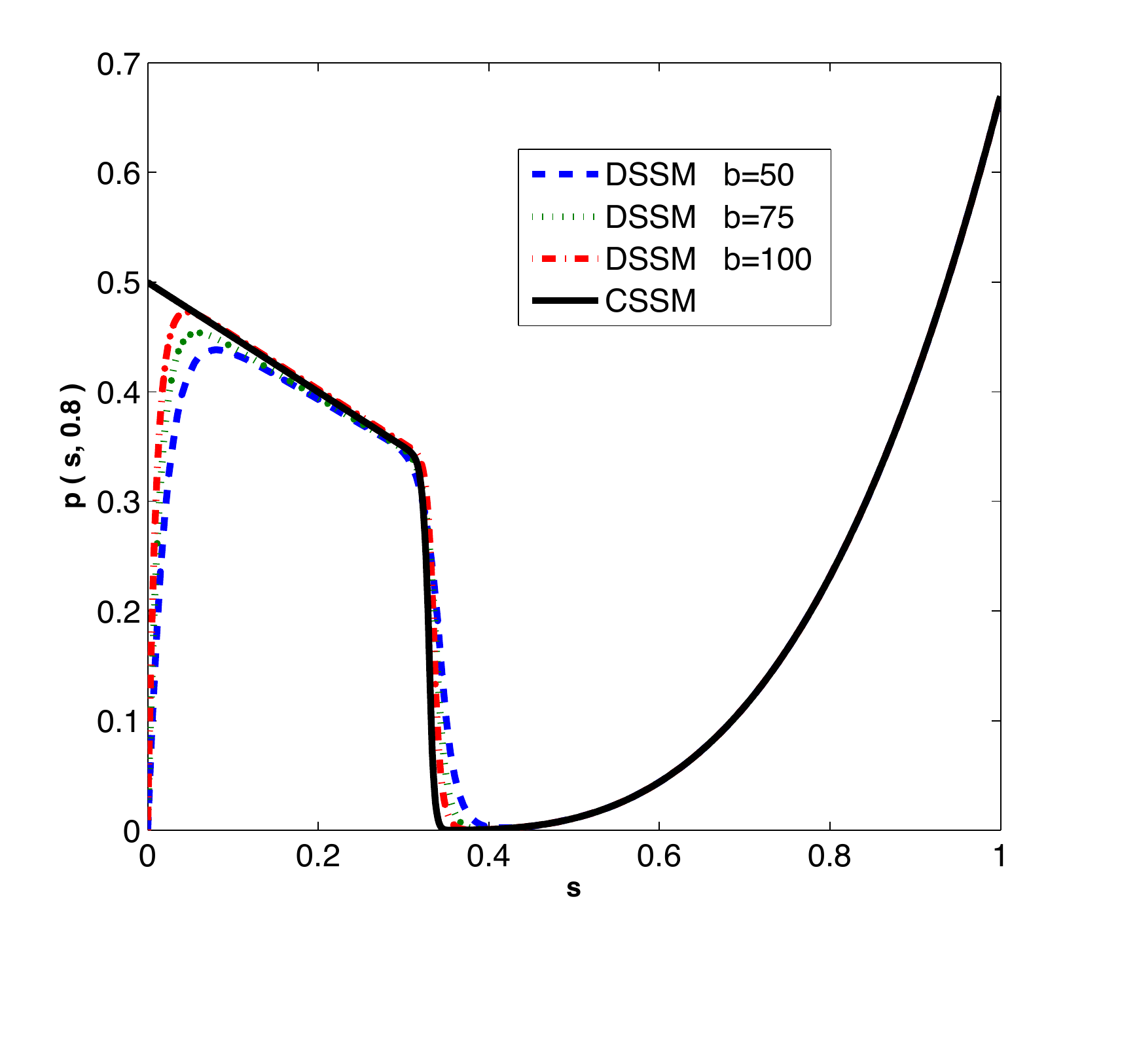}
\caption {\footnotesize The solution $p(s, 0.8)$ for DSSM corresponding to $a=1.01$ and $b=50, 75, 100$ against the
solution $p(s, 0.8)$ for CSSM. }
\label{beta_distribution}
\end{figure}

\subsection{Supercritical Hopf-bifurcation}
We present a ``toy model" here, in which a unique positive steady state looses its stability via Hopf-bifurcation.
This example is further interesting, since as we will see, the net reproduction function is decreasing at the steady
state
(i.e. its derivative is negative) but the steady state is unstable. In fact this is the only case when stability can be
lost via Hopf-bifurcation, since if the model ingredients are such that the derivative of the net reproduction function
is positive then the governing linear semigroup is positive, see e.g. \cite{FH}.  To illustrate the main ideas first we
introduce a simple example for the classical Gurtin-MacCamy-type model, and then we perform numerical simulations to
show that supercritical Hopf-bifurcation occurs in a corresponding distributed states-at-birth model, too.

Let $\gamma\equiv 1$ and the mortality rate $\mu= \mu(s)$. Assume that the survival probability $\pi(s)
=\exp\left\{-\int_0^s \mu(\tau) d\tau\right\}$ is given by:
\begin{equation}
\begin{array}{ll}
\pi(s)=
\left\{
\begin{array}{ll}
1, & s\in[0,{s_c}],\\
0, & s\in [{s_c}, 1],
\end{array}
\right.
\end{array}
\end{equation}
\begin{equation}
\begin{array}{ll}
\beta(s,Q)=
\left\{
\begin{array}{ll}
0, & s\in [0,q)\cup (q+\varepsilon,1],\, Q\in [0,\infty),\\
e^{-Q}\tilde R\varepsilon^{-1}, & s\in [q,q+\varepsilon],\, Q\in [0,\infty),
\end{array}
\right.
\end{array}
\end{equation}
where $\tilde R>1,\,\varepsilon>0$, and $0<q<{s_c}<1$. Note that in this example both the fertility and mortality
functions are discontinuous. With this choice of the survival probability and fertility function
the net reproduction function reads:
\begin{equation*}
R(Q)=\int_q^{q+\varepsilon}e^{-Q}\tilde R\varepsilon^{-1}\,ds=\tilde Re^{-Q}.
\end{equation*}
Hence for any $\tilde R>1$ there is a unique positive steady state with total population size $Q_*=\ln(\tilde R)$.
We have
\begin{equation*}
p_*(0)=\frac{Q_*}{\int_0^1\pi(s)\, ds}={s_c}^{-1}\ln(\tilde R),\quad \text{and} \quad p_*(s)={s_c}^{-1}\ln(\tilde
R)\pi(s).
\end{equation*}
The characteristic equation corresponding to the linearized system at the positive steady state reads (see e.g.
\cite{FH}):
\begin{align}
1=K(\lambda)&=\int_0^1\beta(s,Q_*)\pi(s)e^{-\lambda s}\, ds \nonumber \\
& \quad + \int_0^1\pi(s)e^{-\lambda s}\, ds\int_0^1{s_c}^{-1}\ln(R)\beta_Q(s,Q_*)\pi(s)\, ds\nonumber \\
& = e^{-\lambda q}\frac{1-e^{-\lambda \varepsilon}}{\lambda\varepsilon}-{s_c}^{-1}\ln(\tilde
R)\frac{1-e^{-\lambda{s_c}}}{\lambda}.\label{chareq1}
\end{align}
In the limit as $\varepsilon\to 0$ the characteristic equation \eqref{chareq1} reduces to:
\begin{equation}
1=e^{-\lambda q}-{s_c}^{-1}\ln(\tilde R)\frac{1-e^{-\lambda {s_c}}}{\lambda}.\label{chareq2}
\end{equation}
We look first for pure imaginary roots of the characteristic equation \eqref{chareq2}, i.e. assume that
$\lambda=i\alpha$ for some $\alpha\in\mathbb{R}\setminus\{0\}$. For such an eigenvalue equation \eqref{chareq2} reads:
\begin{equation}
1=(\cos(\alpha q)-i\sin(\alpha q))+i\alpha^{-1}{s_c}^{-1}\ln(\tilde
R)(1-\cos(\alpha{s_c})+i\sin(\alpha{s_c})),\label{chareq3}
\end{equation}
which is equivalent to
\begin{align}
1=&\cos(\alpha q)-\alpha^{-1}{s_c}^{-1}\ln(\tilde R)\sin(\alpha{s_c}),\label{chareq4}\\
0=&-\sin(\alpha q)-\alpha^{-1}{s_c}^{-1}\ln(\tilde R)(\cos(\alpha{s_c})-1).\label{chareq5}
\end{align}
Straightforward calculations show that for $q=\frac{1}{6},\,{s_c}=\frac{1}{2}$ and for $\ln(\tilde R)=\frac{3\pi}{2}$
equations
\eqref{chareq4}-\eqref{chareq5} admit the solution $\lambda=3\pi i$.
Next we would like to show that the pair of purely imaginary eigenvalues $\lambda_{1/2}=\pm 3\pi i$ cross the $y$-axis
to the right. To this end we write ${s_c}^{-1}\ln(\tilde R)=3\pi+r$ and
$\lambda=q+3\pi i$, where $r,q\in\mathbb{R}$. The characteristic equation \eqref{chareq2} reads:
\begin{equation*}
1=-ie^{-\frac{q}{6}}-\frac{r+3\pi}{q+3\pi i}(1-ie^{-\frac{r}{2}}),
\end{equation*}
which leads to
\begin{align*}
3\pi=& -qe^{-\frac{q}{6}}+(r+3\pi)e^{-\frac{r}{2}}, \\
q=& 3\pi e^{-\frac{q}{6}}-r-3\pi \Rightarrow r+3\pi=3\pi e^{-\frac{q}{6}}-q.
\end{align*}
Hence for $r<0\, (\iff {s_c}>0.5)$ small enough the eigenvalue $\lambda=q+3\pi i$ with have a positive real part.
Next we note that the continuous dependence of the eigenvalues on the parameter $\varepsilon$ (see e.g.
\cite[Ch.IV.3.5]{Kato}) implies that for $\varepsilon>0$ small enough the eigenvalue will still have a positive real
part.

Our next numerical example demonstrates, for the first time as far as we know,  that such bifurcation may also occur in
the DSSM. This is somewhat surprising mainly because the integral operator representing the distributed
states-at-birth
may have a smoothing effect, in general. We let $p^{0}=s$ and
$$
\begin{array}{ll}
\gamma = 1,\\
\mu=\frac{160}{(250000s^2-250000s+62505)(0.32\arctan(250-500s)+2)},\\
\beta=\beta_1(s,Q)\beta_2(y),  \label{Hopf-example}
\end{array}
$$
where
$$\begin{array}{ll}
\beta_1(s,Q)=a\exp(-Q)\left(10\arctan(5-1000s)+15.7\right),\\
\beta_2(y)=\frac{1}{\sqrt{2\pi}}\exp{\left(-0.5\left(100(y-1/6+0.005)\right)^2\right)}\exp\left(\frac{3\pi}{2}\right).
\end{array}
$$
Here $a$ is a positive constant.
The dynamics of total population $Q(t)$ for different values of $a$ are shown in Figure 5. Only the maximum and minimum
values of $Q$ are plotted in the bifurcation graph of the dynamics of $Q(t)$ with respect to $a$ .

\begin{figure}[H]\label{comparison_1}
\centering
\includegraphics[height=2.75in, width=3.1in]{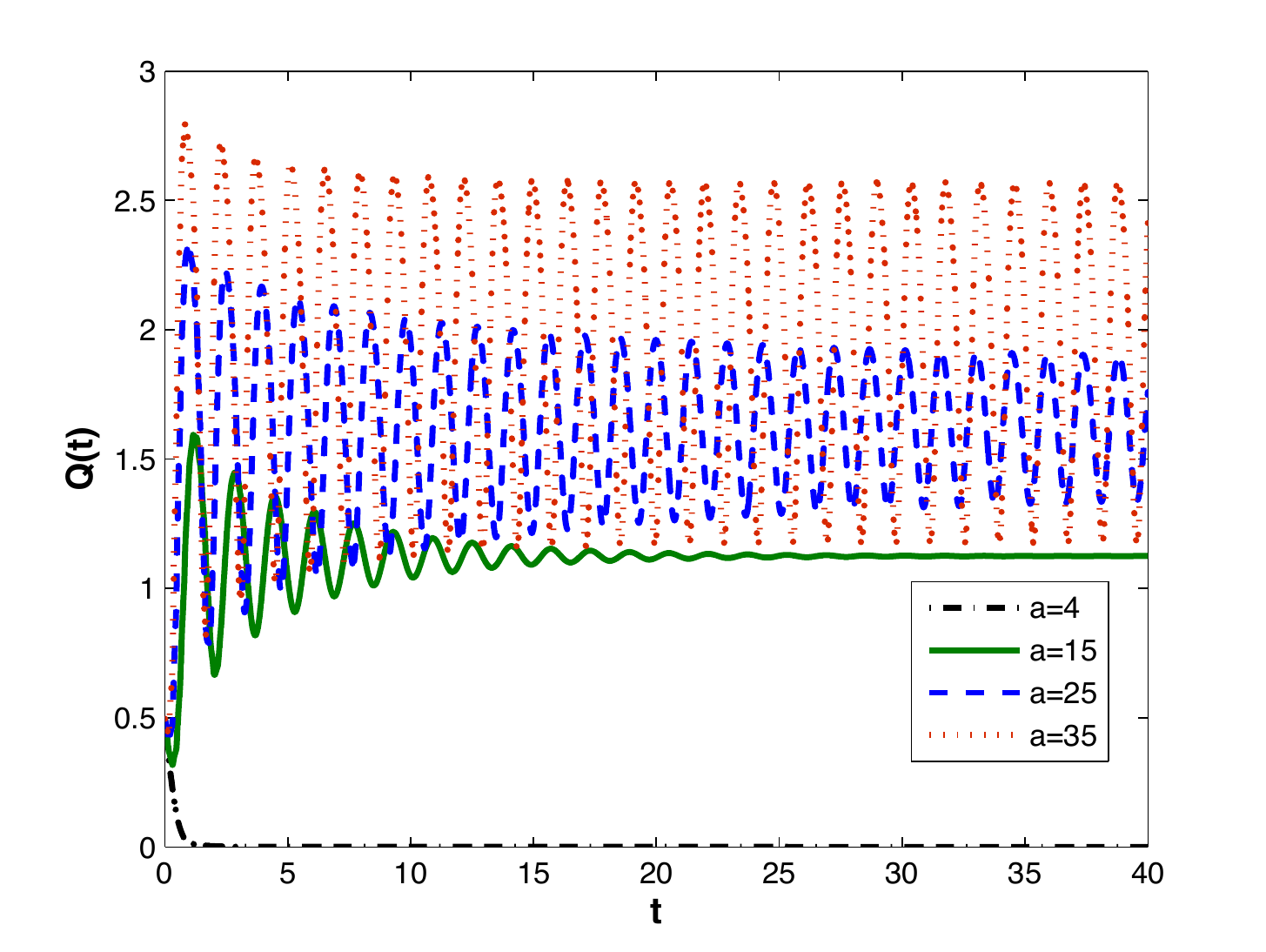}
\includegraphics[height=2.75in, width=3.1in]{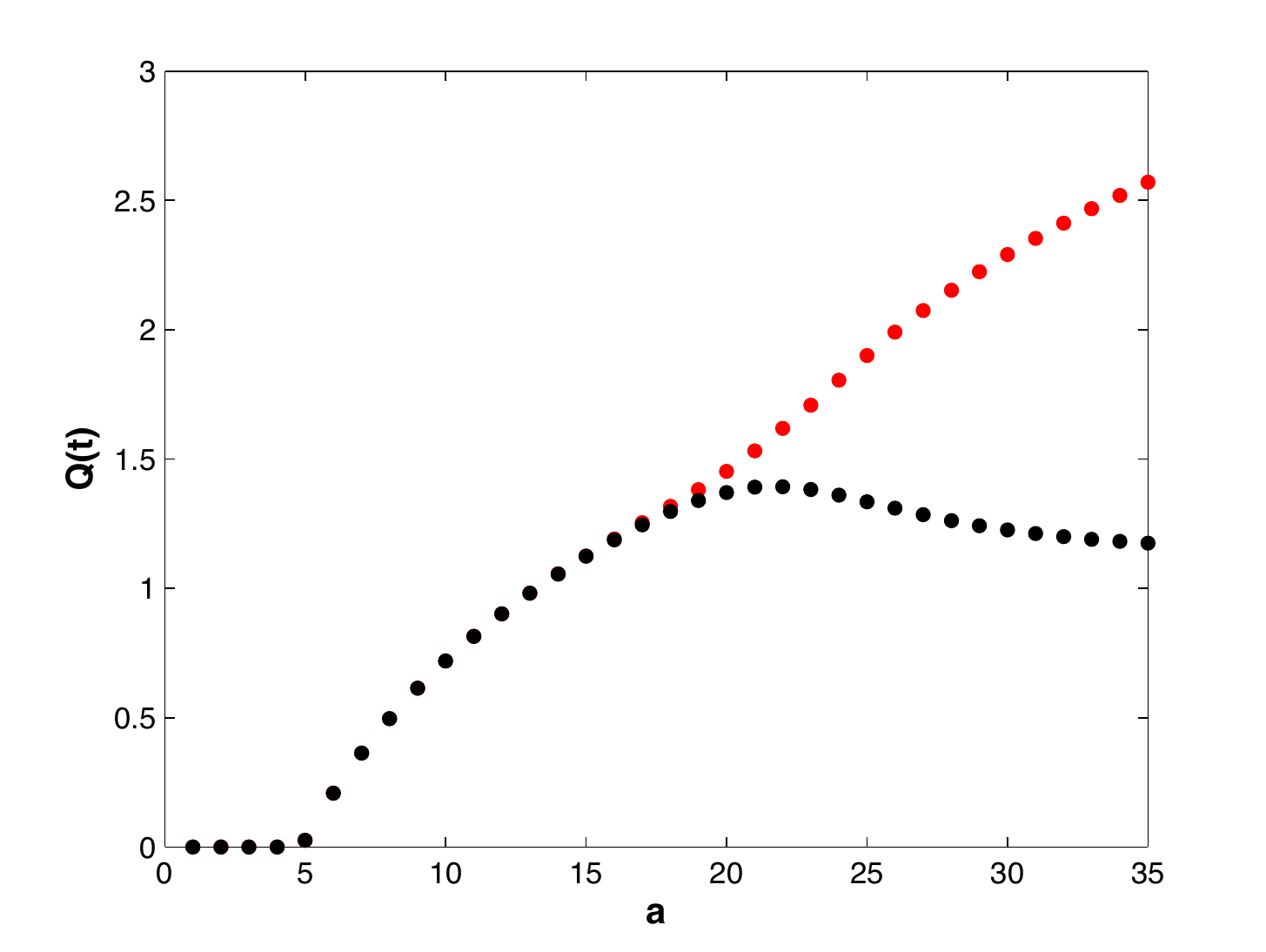}
\caption {\footnotesize Left: A comparison of the total population sizes $Q(t)$ for $a=6, 26, 46$; Right:
Bifurcation graph of $Q$ with respect to parameter $a$.}
\end{figure}

\section{Conclusion}
We have developed a first order upwind scheme and a second order finite difference scheme to approximate the solution
of a size-structured population model with distributed states-at-birth. Convergence of both schemes
to the unique bounded total variation weak solution has been proved. Numerical results are provided to demonstrate the
capability of the numerical methods in resolving smooth as well as discontinuous solutions. For smooth solutions both
schemes achieve the designed order of accuracy. For discontinuous solutions, the second order scheme demonstrates
better accuracy in capturing the discontinuity compared to upwind schemes. The second order scheme is also applied to
the distributed states-at-birth model (DSSM) to show that supercritical Hopf-bifurcation may occur in such models.

We also proved the convergence of the weak solution for the distributed size-structured model (DSSM) in the weak*
topology to that of the classical size-structured model (CSSM) under certain conditions on the fertility function and
used the second order scheme to demonstrate this convergence.\\

\noindent{\bf Acknowledgements:} The work of A.S. Ackleh, X. Li and B. Ma is partially supported by the National Science Foundation
under grant \# DMS-1312963. J. Z. Farkas was supported by a University of Stirling research and enterprise support grant.

\end{document}